\documentclass[11pt,reqno]{amsart}
\usepackage{amsfonts}
\usepackage{amssymb}
\usepackage{color}
\usepackage{amsthm}
\usepackage{mathrsfs}
\usepackage{enumerate}
\usepackage{amsmath}
\usepackage{amstext, amsxtra}
\usepackage{txfonts}
\usepackage[colorlinks, linkcolor=black, citecolor=blue, pagebackref, hypertexnames=false]{hyperref}
\usepackage{graphicx}
\usepackage[symbol]{footmisc}
\usepackage{latexsym}
\usepackage{tikz}
\usepackage{cite}
\allowdisplaybreaks

\usepackage{footmisc} 

\newtheorem{theorem}{Theorem}[section]
\newtheorem{proposition}[theorem]{Proposition}

\newtheorem{lemma}[theorem]{Lemma}

\newtheorem{remark}[theorem]{Remark}

\newcommand\N{\mathbb{N}}
\newcommand\R{\mathbb{R}}

\renewcommand\Re{\operatorname{Re}}

\newcommand{\x}{\times}

\def\R {\mathbb{R}}
\def\N {\mathbb{N}}

\def\Re {\mathfrak{Re\,}}

\def\x{\boldsymbol{x}}

\def\d{{\rm d}}

\def\i{{\rm i}}

\def \and {{\qquad\text{and}\qquad}}

\numberwithin{equation}{section}
\theoremstyle{definition}

\title[observability for the  anharmonic oscillator]
{Quantitative observability for the  Schr\"{o}dinger equation with an anharmonic oscillator}
\author[S. Huang]{Shanlin Huang}

\address{Shanlin Huang
	\newline\indent
	School of Mathematics (Zhuhai), Sun Yat-sen University, Zhuhai 519082, Guangdong, China}
\email{shanlin\_huang@hust.edu.cn}

\author[G. Wang]{Gengsheng Wang}

\address{Gengsheng Wang
\newline\indent
Center for Applied Mathematics, Tianjin University, Tianjin 300072, P.R. China}
\email{wanggs62@yeah.net}

\author[M. Wang ]
{ Ming Wang}

\address{Ming Wang
\newline\indent
School of Mathematics and Statistics, HNP-LAMA, Central South University, Chang-sha, Hunan 410083, P.R. China
}
\email{m.wang@csu.edu.cn}

\begin{document}

\begin{abstract}
	This paper studies the observability inequalities for the Schr\"{o}dinger equation associated with an anharmonic oscillator $H=-\frac{\d^2}{\d x^2}+|x|$.  We build up the observability inequality over an arbitrarily short time interval $(0,T)$, with an explicit expression for the observation constant $C_{obs}$ in terms of $T$, 
	for
	some observable set that has a different geometric structure compared to  those discussed in \cite{HWW}. 
	We obtain the sufficient conditions and the necessary conditions for observable sets, respectively.
	We also present counterexamples to demonstrate that half-lines are not observable sets, highlighting a major difference in the geometric properties of observable sets compared to those of Schr\"{o}dinger operators $H=-\frac{\d^2}{\d x^2}+|x|^{2m}$ with $m\ge 1$.
	
	Our approach is based on the following ingredients: First, the use of an Ingham-type spectral inequality constructed in this paper; second, the adaptation of a quantitative unique compactness argument, inspired by the work of Bourgain-Burq-Zworski \cite{Bour13};
	third, the application of the
	Szeg\"{o}'s limit theorem from the theory of Toeplitz matrices, which provides a new mathematical tool for proving counterexamples of observability inequalities.
\end{abstract}

\maketitle

\section{Introduction}\label{sec1}
\subsection{Problem, aim and motivation.}\label{sec1.1}

For the following    Schr\"{o}dinger equations on $\R^n$:
\begin{equation}\label{equ0.1-w}
	\i \partial_t u(t,x) = H u(t,x),\;\; t\in \mathbb{R}^+:=(0,\infty),\; x\in  \mathbb{R}^n; \quad u(0,\cdot)\in L^2(\mathbb{R}^n),
\end{equation}
(where $H=-\Delta+V(x)$, with $V$  a real-valued function), the observability inequality reads:
\begin{align}\label{equ0.2}
	\int_{\mathbb{R}^n}|u(0,x)|^2\,\mathrm dx\leq C_{obs}\int_0^T\int_E|u(t,x)|^2\,\mathrm dx\, \mathrm dt,\;\mbox{when}\;u\;\mbox{solves}\; \eqref{equ0.1-w},
\end{align}
where $E\subset \mathbb{R}^n$ is a measurable subset, $T>0$ and $C_{obs}:=C_{obs}(E,T)>0$. (Here and in what follows, $C(\cdots)$ stands for a positive constant depending on what is enclosed in the brackets.) Several terminologies related to
\eqref{equ0.2} are as follows:
\begin{itemize}
	\item [($\textbf{D}_1$)] A measurable set $E\subset \mathbb{R}^n$  is called \emph{an observable set at some  time} (or \emph{at any time}) for equation
	\eqref{equ0.1-w}, if for some $T>0$ (or for any $T>0$), there exists a constant $C_{obs}=C_{obs}(T,E)>0$ (the control cost) so that
	\eqref{equ0.2} is maintained. If both $E$ and $T$ satisfy \eqref{equ0.2}, then we say that $E$ is \emph{an observable set at $T$}. 
	
	The constant $C_{obs}$ is called the observable constant (which corresponds to the cost of control in controllability). Given a measurable set $E\subset  \mathbb{R}^n$, we let 
	$$
	T_{min}:=T_{min}(E):=\inf\{T\;:\;\eqref{equ0.2} \;\mbox{holds for}\; T\}.
	$$
	Clearly, the set $E$ is observable at any time if and only if $T_{min}(E)=0$, while if $T_{min}(E)=+\infty$, then for any $T>0$, $E$ is not an observable set at $T$. 
\end{itemize}
The study of observability for  Schr\"{o}dinger equations on $\R^n$ has attracted increasing attention in recent years. For the free Schr\"{o}dinger equation on $\R^n$, T\"{a}ufer has shown in \cite{Taufer} that every periodic open set is \emph{an observable set at any time}. For Schr\"{o}dinger equations with periodic $L^{\infty}$ potentials in $\R^2$,   Le Bal\'{c}h and   Martin have proved in \cite{LM} that every periodic translation of a set of positive measures is \emph{an observable set at any time}. Recently, Prouff has investigated in \cite{Pr} the observability of the   Schr\"{o}dinger equations on $\R^n$ for smooth potentials with subquadratic growth. Sufficient and necessary conditions of observable sets are provided, respectively, through the Hamiltonian flow associated with the Schr\"{o}dinger operator. More recently, for  Schr\"{o}dinger equations with continuous and bounded potentials on $\R$, Su, Sun, and Yuan have proved in \cite{SSY} that every thick set is \emph{an observable set at any time}.

To introduce our main results, motivations and related works, we need the following definitions (where $n=1$):
\begin{itemize}
	\item [($\textbf{D}_2$)] A  measurable set  $E\subset\mathbb{R}$ is said to be \emph{thick},
	if there are constants  $\gamma,L>0$ so that\footnote{Here and below, $|E|$ denotes the 1-dim Lebesgue measure of $E$.}
	\begin{align}\label{equ0.4}
		\left|E \cap [x, x+ L]\right|\geq \gamma L\;\;\mbox{for each}\;\;x\in \mathbb{R}.
	\end{align}

	\item [($\textbf{D}_3$)]  
	A  measurable set  $E\subset\mathbb{R}$ is said to be \emph{weakly thick},
	if
	\begin{align}\label{def-1-3}
		\varliminf_{x \rightarrow +\infty} \frac{|E\cap [-x, x]|}{x} >0.
	\end{align}
	\item [($\textbf{D}_4$)]  A  measurable set  $E\subset\mathbb{R}$ is said to be $\alpha$-thin for some $\alpha>0$ if there exists a constant $C>0$ such that
	\begin{align}\label{def-1-4}
		\varlimsup_{|x| \rightarrow +\infty} |E\cap [x, x+1]|\cdot |x|^{\alpha} <C.
	\end{align}
\end{itemize}

\begin{remark}\label{remark1.1-12-16}
	$(i)$  The  interrelations  among thick sets, weakly thick sets and $\alpha$-thin sets are as follows: 
	\begin{equation*}
		E^c\, \mbox{is  $\alpha$-thin for some $\alpha>0$} \Longrightarrow  E\, \mbox{is  thick}  \Longrightarrow  E\, \mbox{is weakly  thick}, 
	\end{equation*}
	but the above implications do not hold in reverse,   see Remark \ref{rmk3.1}.
	Here 
	and throughout the paper, $E^c$ denotes the complement of $E$.
	
	$(ii)$ The half lines $(0,+\infty)$ and $(-\infty,0)$ are weakly thick sets. 
	
	$(iii)$ The
	$\alpha$-thin sets have been used to study the observability for the linear KdV equation in \cite{WW22}.
	Roughly speaking, it measures the sparsity of the set at infinity.
\end{remark}

This paper intends to explore how potentials affect the observability by studying 
the observability inequality 
for 
the $1$-dim Schr\"{o}dinger equation with anharmonic oscillator:
\begin{align}\label{equ1.1}
	\i \partial_t u(t,x)=\left({-\frac{\d^2}{\d x^2}+|x|}\right)u(t,x), \;\; (t,x)\in \mathbb{R}^+\times \mathbb{R}; \quad u(0,x)\in L^2(\R),
\end{align}
which is exactly \eqref{equ0.1-w} with 
\begin{align}\label{equ0.6}
	H=-\frac{\d^2}{\d x^2}+|x|.
\end{align}
The current study is a continuation of the work initiated in \cite{HWW}, where
we characterize observable sets for
the 1-dim Schr\"{o}dinger equation:
\begin{align}\label{1.8-12-16-w}
	\i \partial_t u(t,x)=\left({-\frac{\d^2}{\d x^2}+|x|^{2m}}\right)u(t,x), \;\; (t,x)\in \mathbb{R}^+\times \mathbb{R}; \quad u(0,x)\in L^2(\R),
\end{align}
which is the equation  \eqref{equ0.1-w} with 
\begin{align}\label{equ0.3}
	H=-\frac{\d^2}{\d x^2}+|x|^{2m},\;\;\mbox{where}\;\; m\in \mathbb{N}:=\{0,1,\dots\}
\end{align}
Let us first  recall the main results in \cite{HWW} in the following remark:
\begin{remark}\label{remark1.2-12-16}
	\noindent $(i)$ When $m=1$ ($m\geq 2$ respectively), $E$ is an observable set at some time (at any time, respectively) if and only if it is weakly thick.
	\noindent $(ii)$ When $m=0$, $E\subset\mathbb{R}$ is an observable set at some time if and only if it is thick.  
	
	It is worth mentioning two facts: First,   ``at some time" in $(ii)$ above has been improved to ``at any time" in \cite{SSY}.
	Second, the potentials in \eqref{equ0.3} 
	exhibit quadratic growth when $m=1$ or superquadratic growth when $m>1$. 
\end{remark}

\textbf{Aim and Motivation.} Our aim is to study the geometry for observable sets, as well as the observable constant $C_{obs}$ and $T_{min}$.  
Our motivations are threefold: First, there are three important quantities in \eqref{equ0.2}: $E$, $T_{min}$, and $C_{obs}$.
Identifying the relationship among them goes a long way toward understanding the essence of observability. It has inspired numerous mathematicians to contribute to this field; see \cite{AM,AR,Bour13,Ha,J90,HWW,KDPhung,Pr} and the references therein.  Our previous work \cite{HWW} made some efforts in this direction. 
Second, the potential in \eqref{equ1.1} is the anharmonic oscillator characterized by subquadratic growth, it is
completely different from those in \cite{HWW}, and poses methodological and technical challenges. 
We hope to reveal new phenomena and methods through the study of the observability inequality for the Schr\"{o}dinger equation 
with such potential. Third, the potentials of anharmonic oscillators play significant roles in many branches of physics and are frequently used in the fields of optics, electromagnetic waves, solid crystals, and quantum field theory (see, for instance, \cite{BW,K,VS}).

\subsection{Main results.}\label{sec1.2}

This paper presents two main theorems.  The first theorem addresses the problem by examining it from a forward perspective, while the second theorem offers an analysis from a reverse viewpoint.

\begin{theorem}\label{thm1.1}
	Let $E\subset \R$ be a measurable set. Suppose that $E^c$ is $\alpha$-thin  with $\alpha>\frac12$.
	Then $E$ is an observable set at any time for \eqref{equ1.1}. Moreover, for each $T\in(0,\frac{1}{2})$, we can choose
	\begin{align}\label{equ1.18,1}
		C_{obs}(T,E)=
		\left\{
		\begin{array}{ll}
			\exp\left\{\exp\big[CT^{- \frac{3}{2}(\frac{1}{2}+\frac{3}{2\alpha-1})}\log \frac{1}{T}\big]\right\}, \quad & \alpha\in(\frac{1}{2},1),\\
			\exp\left\{ \exp\big[CT^{- \frac{21}{4}}(\log \frac{1}{T})^{\frac{11}{2}}\big]\right\}, \quad & \alpha=1,\\
			\exp\left\{ \exp\big[CT^{- \frac{21}{4}}\log \frac{1}{T}\big]\right\}, \quad & \alpha>1,
		\end{array}
		\right.
	\end{align}
	where $C>0$ is a constant depending only on $E$ and $\alpha$.
\end{theorem}

\begin{theorem}\label{thm1.2}
	\noindent $(i)$ If $E$ is an observable set at some time for \eqref{equ1.1}, then $E$ is weakly thick.
	
	\noindent $(ii)$ For each $T>0$, neither  $(0, \infty)$ nor  $(-\infty, 0)$ is 
	{\it an observable set at $T$} for \eqref{equ1.1}.
\end{theorem}
We make the following remarks related to Theorems \ref{thm1.1} and \ref{thm1.2}:
\begin{itemize}
	
	\item [($\textbf{a}_1$)] 
	As far as we are aware, Theorems  \ref{thm1.1}, \ref{thm1.2} seem to be  new. Their proofs are also different from those of previously relevant results,  a point we will elaborate on  in more detail in subsection \ref{sec-compar}.
	Furthermore, Theorem \ref{thm1.1} is robust  under $L^{\infty}$ perturbation, i.e., it remains valid (up to the explicit control cost $C_{obs}$) for the operator $H=-\frac{\d^2}{\d x^2}+|x|+V(x)$, where $V\in L^\infty(\R)$, see Theorem \ref{thm-potent}.

	\item [($\textbf{a}_2$)]
	Theorems  \ref{thm1.1}, \ref{thm1.2}, together with Remarks \ref{remark1.1-12-16}, \ref{remark1.2-12-16}, exhibit the disparities between the observability of the equation \eqref{equ1.1} and the equation \eqref{1.8-12-16-w} (with $m\geq 1$): First, the geometry of the observable sets 
	is different. Second,  the 
	half-line $(0,+\infty)$ (or $(-\infty,0)$)  is \emph{an observable set at $T$} for \eqref{1.8-12-16-w}
	with $m=1$ if and only if $T>\frac{\pi}{2}$ (\cite[Theorem 1.4]{HWW}); it is \emph{an observable set at any time}
	for \eqref{1.8-12-16-w} with $m\ge 2$. However, it is not an observable set for \eqref{equ1.1}, i.e., for each $T>0$, it is not observable at $T$.

	
	\item [($\textbf{a}_3$)]
	The geometry of observable sets for \eqref{equ1.1} and \eqref{1.8-12-16-w} (with $m\geq 1$)
	are closely related to the asymptotic distribution  of the eigenvalues $\{\lambda_k\}_{k=1}^\infty$ of $H$, as well as properties of the corresponding eigenfunctions $\{\varphi_k\}_{k=1}^\infty$ (normalized in $L^2$). Although the precise nature of these dependencies are not fully understood, the following facts provide partial support for this assertion:
	
	{\it Fact 1.} By \cite[Section 4]{HWW}, we have 
	\begin{align}\label{equ1.9}
		\left\{
		\begin{array}{ll}
			\lambda_{k+1}-\lambda_k\rightarrow +\infty,\;\; \mbox{as}\;\; k\rightarrow +\infty, & m\geq 2,\\
			\lambda_{k+1}-\lambda_k= 2, \; \qquad  k\in \mathbb{N}^+,  &  m=1.
		\end{array}
		\right.
	\end{align}
	While when   $m=\frac12$, we have that  (see \eqref{equ2.3} and \eqref{equ2.4})  $\lambda_k\sim k^{\frac23}$ and 
	\begin{align}\label{equ1.10}
		\lambda_{k+1}-\lambda_k\rightarrow 0,\;\; \mbox{as}\;\; k\to \infty, \qquad     m=\frac12.
	\end{align}
	
	{\it Fact 2.}  When  \eqref{equ1.9} holds,  the observability inequality \eqref{equ0.2} is equivalent to the following (see e.g. in \cite[Theorem 1.3]{RTTT}): 
	There exists $C>0$, independent of $k$, so that
	\begin{align}\label{equ1.11}
		\int_E|\varphi_k(x)|^2\d x \geq C\;\;\mbox{ for all }\;\; k\in \mathbb{N}^+.
	\end{align}

	{\it Fact 3.}  When $m=\frac12$,  \eqref{equ1.11} is satisfied with $E=(0,+\infty)$ and $C=\frac12$. This is due to the fact that
	each $\varphi_k(x)$ is either odd or even (see Lemma \ref{lem2.1}).
	However, it follows from  Theorem \ref{thm1.2} that  the observability inequality \eqref{equ0.2}, when applied with $E=(0,+\infty)$,
	does not hold. On the other hand,   taking $u(0, x)=\varphi_k(x)$ in \eqref{equ0.2} leads to \eqref{equ1.11} immediately. 
	Hence,  for the Schr\"{o}dinger equation \eqref{equ1.1},  \eqref{equ1.11} is strictly weaker than  \eqref{equ0.2}.
	This disparity arises because the eigenvalues exhibit sub-linear growth.
	
	We refer to \cite{MR} for a similar situation where the authors construct potentials $V$ on $\mathbb{S}^2$, such that the observability fails for the evolution problem while \eqref{equ1.11} holds for each $L^2$ normalized eigenfunctions.
	We mention that the growth order of the eigenvalues may also have a decisive influence on other control problems such as the rapid stabilization  based on Fredholm backstepping, see \cite{GHXZ}.

\end{itemize}
%
%

\subsection{Strategy and further comments.}\label{sec1.3}
This subsection explains our strategy to prove Theorems \ref{thm1.1} and \ref{thm1.2}, and also compares our approach with those used in related work \cite{SSY,Pr}.
\subsubsection{Strategy to prove Theorem \ref{thm1.1}.}
We outline our strategy in several steps.

{\it Step 1.} We build the following Ingham-type spectral inequality (that is, Proposition \ref{lem-92}): There is $\varepsilon>0$,
$\delta>0$, such that when $E^c$ is $\alpha$-thin (with $\alpha\geq\frac{1}{2}$) and $n\in\mathbb{N}^+$,
\begin{align}\label{equ1.3.0}
	\int_{E}\Big|\sum_{k\in J^\varepsilon_{\alpha}(\lambda_n)}c_k\varphi_k(x)\Big|^2\d x \geq \delta\cdot\sum_{k\in J^\varepsilon_{\alpha}(\lambda_n)}|c_k|^2\;\;\mbox{ for all }\;\; \{c_k\}_{k\in \mathbb{Z}}\in l^2,
\end{align}
where
\begin{align}\label{equ-92-1.1}
	J_{\alpha}^{\varepsilon}(\lambda_n):=\left\{k\in \N^+: |\lambda_k-\lambda_n|<
	\left\{
	\begin{array}{ll}
		\varepsilon \lambda_n^{\alpha-\frac{1}{2}}, \quad & \alpha\in[\frac{1}{2},1),\\
		\varepsilon \lambda_n^{\frac{1}{2}}(\log\lambda_n)^{-1}, \quad &\alpha=1,\\
		\varepsilon \lambda_n^{\frac{1}{2}} , \quad &\alpha>1.
	\end{array}
	\right.
	\right\}.
\end{align}
It can be viewed as a variable coefficient version (associated with the eigenelements of $H$) of the  classical Ingham inequality, 
which states that if
the strictly increasing sequence $\{\lambda_k\}_{k\in\mathbb{Z}}$ of real numbers satisfies the `gap' condition
$
\lambda_{k+1}-\lambda_k\ge \gamma,\;\;\forall\; k\in \mathbb{Z}
$
for some fixed $\gamma>0$.
Then, for all $T>\frac{2\pi}{\gamma}$, there exists some positive constant $\delta>0$ depending only on $\gamma$ and $T$ such that (see e.g. in  \cite{In} and \cite[p. 162]{KL})
\begin{align*}\label{equ1.3.2}
	\int_{0}^{T}\left|\sum_{k\in \mathbb{Z}}c_ke^{\i\lambda_kx}\right|^2\d x \geq \delta\cdot\sum_{k\in \mathbb{Z}}|c_k|^2,\;\; \mbox{for all}\;\;\{c_k\}_{k\in \mathbb{Z}}\in \textit{l}^2.
\end{align*}
The proof of \eqref{equ1.3.0}   combines two main ingredients: establishing uniform lower bounds for diagonal terms and obtaining upper bounds for off-diagonal terms. The lower bounds hold if and only if $E$ is \emph{weakly thick} (see Lemma \ref{lem-low-h}); while for the upper bounds, we require the assumption that $E^c$ is $\alpha$-thin (see Lemma \ref{lem-91}).

It should be mentioned that by \eqref{equ1.3.0}, we can deduce that $E$ is \emph{an observable set at some time}.
However, to strengthen this statement from ``at some time" to ``at any time", additional steps and arguments are needed.

{\it Step 2.} We use
\eqref{equ1.3.0} to obtain a relaxed observability inequality at any time (that is, Proposition \ref{lem-ob-comp}).  The phrase ``relaxed observability inequality" refers to a weaker observability inequality which involves an error  term. This is proved in a similar way to an abstract argument of Burq-Zworski \cite[Theorem 4]{BZ2004} (see also Miller \cite{Miller05}), and finally reduces to a uniform resolvent estimate. The resolvent estimate is obtained by applying the spectral inequality established in Step 1.

{\it Step 3.} We use the aforementioned  relaxed observability inequality to show that $E$ is \emph{an observable set at any time}. In this step, we borrow an idea from the quantitative compactness arguments of uniqueness due to Bourgain-Burq-Zworski \cite[Theorem 4]{Bour13}, where an observability inequality was established for Schr\"{o}dinger equations with rough potentials on 2-tori. These arguments can be seen as a quantitative extension of the classical uniqueness-compactness argument proposed by Bardos–Lebeau–Rauch \cite{BLR}.  With the help of these arguments, we are able to effectively glue the high-frequency and low-frequency estimates together. Furthermore, thanks to explicit information  regarding the distribution of eigenvalues,   we are also able to derive  an explicit expression of $C_{obs}(T,E)$ in terms of $T$, which may be of independent interest. This is achieved by using Salem and Nazarov's inequality (Lemma \ref{lem-Salem}-\ref{lem-Nazarov}) on triangle polynomials; see Section \ref{sec3.3}.

\subsubsection{Idea to prove  Theorem \ref{thm1.2}.}
The proof of the conclusion $(i)$ relies on the asymptotic pointwise behavior of the high-frequency eigenfunctions, which
is achieved by using the properties of Airy functions (see \eqref{equ2.5} and \eqref{equ2.8}).  Since the eigenfunction decays polynomially in the classical allowed region but decays exponentially in the classically forbidden region, we shall deal with them separately.

For the conclusion $(ii)$, we are able to calculate the exact values of the following integrals in terms of the Airy kernel (see \eqref{equ2.12}):
$$
a_{k,j}:=\int_{0}^{\infty}{\varphi_{k}(x)\cdot\varphi_{j}(x)\,\mathrm dx} \quad\text{or}\quad\,\,
\int_{-\infty}^{0}{\varphi_{k}(x)\cdot\varphi_{j}(x)\,\mathrm dx},\,\quad k\ne j\in J_{1/2}^{\varepsilon}(\lambda_n)
$$
(Noticing that $E=(0,+\infty)$ or $(-\infty,0)$.)
They  represent contributions of off-diagonal terms on the left-hand side of \eqref{equ1.3.0}. Then, by taking advantage of the properties of the zeros of the Airy function $Ai(\cdot)$ and of its derivatives $Ai'(\cdot)$,
we can explicitly demonstrate the interaction between two distinct (high-frequency) eigenfunctions, which reflects
the loss of orthogonality in such regions.
Moreover, we observe that the associated gram matrices $(a_{k,j})$ (see \eqref{equ3.3}) are symmetric Toeplitz matrices. With this crucial observation in mind, we then apply the classical Szeg\"{o}'s limit theorem (see Lemma \ref{lem2.2}) to derive the conclusion $(ii)$.

\subsubsection{Comparison with related works}\label{sec-compar}
It is worth mentioning that the method of proving Theorem \ref{thm1.1} is quite different from the methods presented in \cite{SSY,Pr,HWW}, which do not appear to be directly applicable to our case. 

In \cite{Pr}, observability is established by using semi-classical analysis techniques, which require the potentials to be smooth and grow sub-quadratic at infinity. More precisely,
denote by $\left(\phi_{0}^{t}\right)_{t\in \mathbb{R}}$ the Hamiltonian flow associated with the symbol $p_{0}(x,\xi)=V_0(x)+\frac{1}{2}|\xi|^{2}$. For any Borel set $\omega \subset \mathbb{R}^{n}$, 
and  any $T>0$, let
$$
\mathfrak{K}_{p_{0}}^{\infty}(\omega,T)=\liminf_{\rho\rightarrow\infty}\int_{0}^{T}1_{\omega\times \mathbb{R}^{n}}\left(\phi_{0}^{t}(\rho)\right)\d t=\liminf_{\rho\rightarrow\infty}\left|\left\{t\in(0,T):(\pi\circ\phi_{0}^{t})(\rho)\in\omega\right\}\right|.
$$
It is proved in \cite[Theorem 1.3]{Pr} that when the potential $V_0$ is smooth and grows subquadraticly, if there exists \(T_{0}>0\) such that
\begin{equation}\label{eq-prouff}
	\mathfrak{K}_{p_{0}}^{\infty}:=\mathfrak{K}_{p_{0}}^{\infty}(\omega,T_{0})>0,    
\end{equation} 
then for any compact set $K\subset \mathbb{R}^{n}$ and any $T>T_{0}$, the following observability inequality holds:
$$
\exists C>0:\forall u\in L^{2}(\mathbb{R}^{n}),\qquad\left\|u\right\|_{L^{2}(\mathbb{R}^{n})}^{2}\leq C\int_{0}^{T}\left\|e^{-\i tP}u\right\|_{L^{2}(\omega_{R}\backslash K)}^{2}\d t,
$$
where $\omega_{R}$ denotes the thickened set
$\omega_{R}=\bigcup_{x\in\omega}B_{R}(x)$  for some $R>0$ depending on $\mathfrak{K}_{p_{0}}^{\infty}$, and $P=-\Delta+V(x)$, here $V$ and $V_0$ have the same principal symbol, namely ($0\leq m\leq 1$)
$$
\forall \alpha \in \N^n, \exists C_\alpha>0: \forall x \in \R^n, \quad\left|\partial_x^\alpha\left(V-V_0\right)(x)\right| \leq C_\alpha\langle x\rangle^{2 m-1-|\alpha|} .
$$

In general, the method in \cite{Pr} belongs to the time-dependent approach. However, we adopt a resolvent approach and does not require the potential $V$ to be smooth. 
It has the following advantages: (i) the observation region is allowed to be a measurable set, instead of being an open set; (ii) we impose geometric conditions directly on the set itself,  rather than on the dynamical properties of the Hamiltonian flow; (iii) the resolvent estimate is robust and it can handle $L^{\infty}$ perturbations effectively, see Subsection \ref{app-per}.

In \cite{SSY}, the observability inequality associated with thick sets is proved for 1-dim Schr\"{o}dinge equation over $\R$ with real-valued,
bounded continuous potential. The proof there
is based on three key components: (i)  a spectral inequality that addresses the low frequencies of 1-dim Schr\"{o}dinger operators with bounded potentials; (ii) a resolvent estimate for the operator $H=-\partial_x^2+V$ at high frequencies; and
(iii) a quantitative glue argument, which is inspired by the general strategy introduced by Phung \cite{KDPhung}.  It is noteworthy that  the resolvent estimate used in \cite{SSY} reads
\begin{equation}\label{resol-s}
	\exists C>0 : \forall f \in H^2(\mathbb{R}), \forall \mu>\mu_0, \|f\|_{L^2(\mathbb{R})}^2 \leq \frac{C}{\mu} \|(H - \mu)f\|_{L^2(\mathbb{R})}^2 + C\|f\|_{L^2(\omega)}^2,
\end{equation}
where $\mu_0>0$ is a constant and $\omega$ denotes a thick set. Estimate \eqref{resol-s} can be deduced through the direct application of real analysis techniques, as proved in  \cite[Lemma 4.1]{SSY}, or  by invoking the Logvinenko-Sereda uncertainty principle \cite{Gre}. We establish a similar estimate for $H=-\partial_x^2+|x|$ (see Proposition \ref{lem-92-resolve}). However, in contrast to the case that $V$ is bounded, it appears that the aforementioned tools used in the resolvent estimate \eqref{resol-s} cannot be directly applied to derive  Proposition \ref{lem-92-resolve}. This limitation arises because the potential $|x|$ is unbounded, and there is a fundamental difference in the spectral structures between $H_0=-\partial_x^2$ (with a spectrum $\sigma(H_0)=[0, \infty)$) and $H=-\partial_x^2+|x|$ (whose spectrum has only discrete eigenvalues; see \eqref{equ2.1}-\eqref{equ2.2} below).

In our previous work \cite{HWW}, we established the observability inequality for \eqref{1.8-12-16-w}. 
There, we took advantage of the following property: the set $E$ is \emph{an observable set at any time} (\emph{at some time}, respectively) for \eqref{1.8-12-16-w} if and only if \eqref{equ1.11} holds.
However,  \eqref{equ1.11} is only a necessary but not sufficient condition for $E$ to be an observable set for \eqref{equ1.1}. 
Furthermore, we must take into account the contribution of the interactions between distinct eigenfunctions, i.e.
$$
\int_E \varphi_j(x) \varphi_k(x) \d x.
$$
Besides, for $m>1$, the quantitative compactness uniqueness arguments of Bourgain-Burq-Zowski are not needed. This is due to the relation \eqref{equ1.9} and the abstract result \cite[Corollary 6.9.6]{TW}.

\subsection{Plan of the paper}

The remainder of the paper is organized as follows.
In Section \ref{sec2}, we present some auxiliary results and tools that will be used in our subsequent discussions;
Section \ref{sec3}  is dedicated to proving  Theorem \ref{thm1.1};
In Section \ref{sec4}, we give the proof of Theorem \ref{thm1.2}. Finally, Appendix \ref{app-001}-\ref{app-002} contain proofs for two intermediate and technical results.
\section{Auxiliary  results and tools}\label{sec2}
\subsection{Spectral properties of the operator $H=-\frac{\d^2}{\d x^2}+|x|$}\label{sec2.1}
We start with recalling several known facts on the spectral theory of the anharmonic oscillator
$H=-\frac{\d^2}{\d x^2}+|x|$ on $L^2(\mathbb{R})$, which can be found in \cite{CHS,Ga,VS}. 

\noindent{\it Fact 1.} $H$ has a compact resolvent and  the spectrum $\sigma(H)=\{\lambda_k\}_{k=1}^\infty$, with
\begin{align}\label{equ2.1}
	1< \lambda_1<\cdots<\lambda_{2k-1}<\lambda_{2k}\rightarrow +\infty,\quad\,\, k\rightarrow +\infty.
\end{align}
Moreover, each eigenvalue $\lambda_k$ is simple, that is, the dimension of each eigenspace is one.

\noindent {\it Fact 2.} The eigenvalues of $H$ are determined by (see e.g. in  \cite[p.152]{VS})
\begin{align}\label{equ2.2}
	Ai(-\lambda_{2k-1})=0,\quad\,\,Ai'(-\lambda_{2k})=0,\quad\,\ k=1,2,\ldots,
\end{align}
where $Ai(\cdot)$ is the Airy function  given by
\begin{align*}
	Ai(x)=\frac{1}{2\pi}\int_{-\infty}^{\infty}{e^{\i(x\cdot t+t^3/3)}\,\mathrm dt}.
\end{align*}
{\it Fact 3.} It follows from  \cite[Fact 2.7]{Ga} that
\begin{align}\label{equ2.3}
	\lambda_k=\left(\frac{3\pi}{4}\cdot k\right)^{\frac23}+o(k^{\frac23})
\end{align}
and 
\begin{align}\label{equ2.4}
	\frac{\pi}{2}\cdot\lambda^{-\frac12}_{k+1}\leq \lambda_{k+1}-\lambda_k\leq \frac{\pi}{2}\cdot\lambda^{-\frac12}_{k}.
\end{align}
{\it Fact 4.} Let $\varphi_k$ be the $L^2$ normalized eigenfunction corresponding to $\lambda_k$.
By  \cite[Theorem 2.6]{Ga} (or \cite[p.153]{VS}), we have
\begin{equation}\label{equ2.5}
	\varphi_k(x)=\begin{cases}
		A_k\cdot Ai(x-\lambda_k),\;\;\qquad \qquad \quad \text{if}\,\, x\ge 0, \\
		(-1)^{k+1}\cdot A_k\cdot Ai(-x-\lambda_k),\;\;\;\;\, \text{if}\,\,x< 0,
	\end{cases}
\end{equation}
where
\begin{align}\label{equ2.6}
	A_k=\left(2\int_{-\lambda_k}^{\infty}{|Ai(x)|^2\,\mathrm dx}\right)^{-\frac12}.
\end{align}

The next lemma gives some properties on $\varphi_k(x)$ which will be used later. 
\begin{lemma}\label{lem2.1}
	\noindent $(i)$ Each eigenfunction $\varphi_k$ is even or odd. Precisely, we have
	\begin{align}\label{equ2.7}
		\varphi_{k}(x)=(-1)^{k+1}\varphi_{k}(-x),\,\,\,\,\,x\in\mathbb{R}.
	\end{align}

	\noindent $(ii)$  When $x>0$, we have
	\begin{equation}\label{equ2.8}
		\varphi_{k}(x)=\begin{cases}
			a_{k}\cdot(\lambda_k-x)^{-\frac14}\left(\sin(\frac23(\lambda_k-x)^{\frac32}+\frac{\pi}{4})+R_k(x)\right),\,\,\,0<x<\lambda_k- \delta,\\
			O(\lambda_k^{-\frac{1}{4}}),\,\,\,\,\,\,\,\qquad\qquad\qquad\qquad\qquad\qquad \qquad\lambda_k- \delta\leq x\leq \lambda_k+ \delta,\\
			a_{k}\cdot(x-\lambda_k)^{-\frac14}e^{-\frac23(x-\lambda_k)^{\frac32}}\left(1+R_k(x)\right),\qquad\qquad\quad\,\,x>\lambda_k+ \delta,
		\end{cases}
	\end{equation}
	where $\delta>0$ is a small but fixed constant  independent of $\lambda_k$ and $x$. Moreover,
	\begin{align}\label{equ2.9}
		|a_k|\sim \lambda_k^{-\frac{1}{4}},\,\,\,\,R_k(x)=O\left(|x-\lambda_k|^{-\frac32}\right).
	\end{align}
\end{lemma}
Here and in what follows, given the sequences of numbers $\{\alpha_k\}$ and $\{\gamma_k\}$, by $\alpha_k\sim \gamma_k$, we mean that
there are $C_1>0$ and $C_2>0$ so that $C_1|\gamma_k|\leq |\alpha_k|\leq C_2|\gamma_k|$ for all $k$, while by $\alpha_k=O(\gamma_k)$,
we mean that
there is $C_3>0$ so that $|\alpha_k|\leq C_3|\gamma_k|$ for all $k$.

\begin{proof}
	Conclusion $(i)$ follows directly from \eqref{equ2.5}. 
	
	To show the conclusion $(ii)$, we recall the following asymptotic expression of  $Ai(\cdot)$
	(see, e.g., \cite[Chapter 2]{VS}):
	\begin{equation}\label{equ2.10}
		Ai(-x)=\begin{cases}
			\pi^{-\frac12}\cdot x^{-\frac14}\left(\sin{\frac23x^{\frac32}}+R_k(x)\right),\qquad \qquad\, x>\delta,\\
			O(1),\,\,\,\,\,\,\,\qquad\qquad\qquad\qquad \qquad -\delta\leq x\leq\delta,\\
			\pi^{-\frac12}\cdot (-x)^{-\frac14}e^{-\frac23x^{\frac32}}\left(1+R_k(x)\right),\qquad\quad x<-\delta,
		\end{cases}
	\end{equation}
	where $\delta>0$ is a small but fixed constant, and $R_k(x)=O(|x|^{-\frac32})$.
	Meanwhile, by \eqref{equ2.10}, we have 
	\begin{align}\label{equ2.10.1}
		\int_{-\lambda_k}^{\infty}{Ai^2(x)\,\mathrm dx}=\int_{-\lambda_k}^1{Ai^2(x)\,\mathrm
			dx}+\int_{1}^{\infty}{Ai^2(x)\,\mathrm dx}\sim \lambda_k^{\frac12}.
	\end{align}
	This, together with \eqref{equ2.6}, implies that
	\begin{equation}\label{equ2.10.2}
		A_k\sim \lambda_k^{-\frac14}.
	\end{equation}
	
	Now, \eqref{equ2.8} and \eqref{equ2.9} follow from \eqref{equ2.5}, \eqref{equ2.10} and \eqref{equ2.10.2}. This
	completes the proof of Lemma \ref{lem2.1}.
\end{proof}

We will also use the following properties on the Airy function and $A_k$ (given by \eqref{equ2.6}) later. 
First, we have
\begin{equation}\label{equ2.11}
	\int_{-\lambda_k}^{\infty}{Ai^2(x)\,\mathrm dx}=\lambda_k\cdot Ai^2(-\lambda_k)+Ai'^2(-\lambda_k)=\begin{cases}
		\lambda_k\cdot Ai^2(-\lambda_k),\,\,\,\text{if $k$ is even},\\
		Ai'^2(-\lambda_k),\quad\,\,\,\,\,\text{if $k$ is odd}.
	\end{cases}
\end{equation}
In fact, the first equality follows from the integration by parts and the fact that $Ai''(x)=x\cdot Ai(x)$, while the second follows from \eqref{equ2.2}.
Second, we have that 
(see e.g. in \cite{TW94}):
\begin{equation}\label{equ2.12}
	Ai(x,y)=
	\begin{cases}
		\frac{Ai(x)Ai'(y)-Ai(y)Ai'(x)}{x-y}, & x\ne y, \\
		Ai'(x)^2-xAi(x)^2 & x=y. \\
	\end{cases}
\end{equation}
Third, by
\cite[p. 165]{TW94}, we also have
\begin{align}\label{equ2.13}
	Ai(x,y)=\int_0^{\infty}{Ai(t+x)Ai(t+y)\,\mathrm dt}.
\end{align}

\subsection{Spectral property of Hermitian Toeplitz matrix and Szeg\"{o}'s theorem}\label{sec2.2}
We recall that a Toeplitz matrix is an $n \times n$ matrix $T_n = (a_{k,j})_{k,j =1}^n$
satisfying
$$a_{k,j} = a_{k-j},\qquad 1\leq k, j\leq n,$$
i.e., a matrix of the form
\begin{align}\label{equ2.14-1}
	T_n=T_n[a_{1-n},\cdots,a_{-1}, a_0, a_1,\cdots, a_{n-1}]:=\begin{bmatrix}
		a_0 & a_{-1} & a_{-2} & \ldots &a_{1-n} \\\
		a_{1} & a_0 & a_{-1} & \ddots & \vdots \\\
		a_{2} & a_{1} & \ddots & \ddots & a_{-2} \\\
		\vdots & \ddots & \ddots  & a_0 & a_{-1}\\\
		a_{n-1} & \ldots  & a_{2} & a_{1} & a_0
	\end{bmatrix}.
\end{align}
Given a Lebesgue integrable complex valued function $f\in L^1[-\pi, \pi]$, we let
\begin{align}\label{equ2.15}
	T_n(f):=T_n[a_{1-n},\cdots,a_{-1}, a_0, a_1,\cdots, a_{n-1}],
\end{align}
with 
\begin{align}\label{equ2.16}
	a_k=\frac{1}{2\pi}\int_{-\pi}^{\pi}{f(t)e^{-\i kt}\,\mathrm dt},\;\;k=0, \pm 1, \cdots, \pm (n-1).
\end{align}
We say that $T_n(f)$ is generated by $f$. We call $T_n(f)$ 
the $n$-th finite section of the Toeplitz matrix associated with the function $f$ and $f$ the symbol of the matrix $T_n(f)$.
One can directly check that if $f$ is real valued, then $\bar{a}_k=a_{-k}$, i.e.,  $T_n(f)$ is Hermitian, and that if $f$ is an even and real-valued function, then $T_n(f)$ is a real matrix.

The asymptotic behavior of the eigenvalues of $T_n(f)$ as $n\rightarrow\infty$  has been
thoroughly studied by mathematicians and physicists for a long time.
In this paper,
we shall use  the following fundamental result (due to  Szeg\"{o}, see \cite[p. 65]{GS}), which describes the  asymptotic behavior of the extreme eigenvalues of the eigenvalues
of the Hermitian Toeplitz matrix
$T_n$  as $n\rightarrow\infty$.
\begin{lemma}[\textbf{Szeg\"{o}'s limit theorem}]\label{lem2.2}
	Assume that $f\in L^1[-\pi, \pi]$ is real valued. Let
	\begin{align}\label{equ2.17}
		\lambda_{\text{min}}=\lambda_1(T_n)\leq \cdots\leq \lambda_n(T_n)=\lambda_{\text{max}}
	\end{align}
	denote the eigenvalues of the Hermitian Toeplitz matrix $T_n(f)$ given by \eqref{equ2.15} and \eqref{equ2.16}.  If we set
	\begin{align}\label{equ2.18}
		m_f=\inf_{x\in(-\pi,\, \pi)} f(x),\qquad M_f=\sup_{x\in(-\pi,\,\pi)} f(x).
	\end{align}
	Then for all $n\ge 1$,
	$$m_f\leq \lambda_1(T_n)\leq \lambda_n(T_n)\leq M_f.
	$$
	Moreover, one has
	\begin{align}\label{equ2.19}
		\lim_{n\rightarrow \infty}\lambda_1(T_n)=m_f,\qquad \lim_{n\rightarrow \infty}\lambda_n(T_n)=M_f.
	\end{align}
\end{lemma}

\subsection{Salem and Nazarov inequality}\label{sec2.3}
We will present two well-known results on trigonometric polynomials that will be used in the proof of Theorem \ref{thm1.1}. 
Given $n\in \mathbb{N}^+$, we let 
\begin{align}\label{equ-exp-p}
	p(t)=\sum_{k=1}^n c_ke^{\i\lambda_kt},\;\;t\in\R,
\end{align}
where 
$c_k\in \mathbb{C},\lambda_k\in \R$ and $k=1,2,\cdots,n$. It is an exponential polynomial.
Without loss of generality, we can assume that $\lambda_1<\lambda_2<\cdots<\lambda_n$. The following Salem inequality is stated in Nazarov's work \cite[Sec. 1.1]{Nazarov}. We refer to \cite[p.222]{Zy} for a detailed proof.
\begin{lemma}[Salem inequality]\label{lem-Salem}
	Assume that the sequence $\{\lambda_k\}_{1\leq k\leq n}$ has a uniform gap $\Delta>0$, namely,
	$$
	\inf_{1\leq k\leq n-1}\lambda_{k+1}-\lambda_k\geq \Delta.
	$$
	Let $I\subset \R$ be an interval with the length $|I|\geq 4\pi/\Delta$. Then the exponential polynomial \eqref{equ-exp-p} satisfies
	\begin{equation}\label{equ-salem}
		\sum_{k=1}^n|c_k|^2\leq \frac{4}{|I|}\int_I|p(t)|^2\d t.
	\end{equation}
\end{lemma}

Inequality \eqref{equ-salem} can be viewed as a finite sum analogue of Ingham's inequality (see \cite{In} and \cite[p. 162]{KL}) with explicit constants. We remark that inequalities akin to \eqref{equ-salem}, which do not specify explicit constants, and are also referred to Ingham's inequality in many literatures.

\begin{lemma}[Nazarov inequality \cite{Nazarov}]\label{lem-Nazarov}
	Let $I\subset \R$ be an interval and $E\subset I$ be a measurable set. Then for all exponential polynomials of the type \eqref{equ-exp-p},
	$$
	\int_I|p(t)|^2\d t \leq \left( \frac{A|I|}{|E|}\right)^{2n-1}\int_E|p(t)|^2\d t,
	$$
	where $A>0$ is an absolute constant.
\end{lemma}

The Nazarov inequality provides a quantitative unique continuation property for exponential polynomials, which is a generalization of Remez's inequality for polynomials.

\section{Proof of Theorem \ref{thm1.1}}\label{sec3}
This section aims to prove Theorem \ref{thm1.1}. The proof is based on the Ingham-type spectral inequality  (Proposition \ref{lem-92}) in Subsection \ref{sec3.1}
and the relaxed observability inequality  (Proposition \ref{lem-ob-comp}) in Subsection \ref{sec3.2}.

\subsection{Ingham-type spectral inequality}\label{sec3.1}
The main result of this subsection is the following Ingham-type spectral inequality:

\begin{proposition}\label{lem-92}
	Assume that $E^c$ is $\alpha$-thin with $\alpha\geq  \frac{1}{2}$. Then there exist  $\varepsilon>0$ and $\delta>0$ such that
	for each $n\in \N^+$,
	\begin{align}\label{equ-92-0}
		\int_E\left|\sum_{k\in  J_{\alpha}^{\varepsilon}(\lambda_n)}c_k\varphi_k(x)\right|^2\d x \geq \delta \sum_{k\in  J_{\alpha}^{\varepsilon}(\lambda_n)}|c_k|^2\;\;\mbox{for all}\;\; \{c_k\}\in l^2,
	\end{align}
	where $J_{\alpha}^{\varepsilon}(\lambda_n)$ is given by \eqref{equ-92-1.1}. 
\end{proposition}
In order to prove Proposition \ref{lem-92}, we need the following two lemmas which concern the contributions of diagonal terms and off-diagonal terms, respectively.
\begin{lemma}[Lower Bounds for Diagonal Terms]\label{lem-low-h}
	The following statements are equivalent:
	
	\noindent $(i)$ The set $E\subset \R$ is weakly thick, i.e., \eqref{def-1-3} holds.
	
	\noindent $(ii)$ There exists an absolute constant $D_1>0$ independent of $k$ such that
	\begin{align}\label{low-h}
		\int_E|\varphi_k(x)|^2\d x \geq D_1\;\;\mbox{ for all }\;\; k.
	\end{align}
\end{lemma}
\begin{proof}
	Its proof is given in the appendix \ref{app-001}.
\end{proof}

\begin{remark}\label{rmk3.1}
	We have the following two conclusions: $(i)$   If $E^c\subset\mathbb{R}$ is $\alpha$-thin for some $\alpha>0$, then $E$ is thick, i.e., \eqref{equ0.4} holds. But the converse is not true. $(ii)$    If $E\subset\mathbb{R}$ is thick, then it is weakly thick. But the converse is not true. 
	
	In fact, the conclusion $(ii)$ has been proved in \cite[Section 4.3]{HWW}. We now prove the conclusion \noindent $(i)$. Suppose that
	$E^c$ is $\alpha$-thin. Then there exists  $L_1>0$ such that
	\begin{align}\label{equ4.34}
		\left|E^c\bigcap [x, x+1]\right| <C_1|x|^{-\alpha}, \quad\,\,\,\mbox{for all} \,\,|x|\ge L_1,
	\end{align}
	where $C_1>0$ is a constant independent of $x$. Now we choose $L_2\ge L_1$ so that $C_1\cdot L_2^{-\alpha}<\frac12$.
	Then it follows from \eqref{equ4.34}  that
	$|E\bigcap [x, x+1]|\ge \frac12$ for all $|x|\ge L_2$. Let $L=4L_2$. Then, a direct computation yields
	$|E\bigcap [x, x+L]|\geq \frac{L}{4}$ for all $x\in \mathbb{R}$. Thus, $E$ is thick. In contrast, we let $E=\bigcup_{n\in\mathbb{Z}}(n, n+\frac{1}{2})$. It is clear that $E$ is thick. However, for each  $n\in \mathbb{Z}$,  we have $|E^c\bigcap [n, n+1]|=|[n+\frac12, n+1]|=\frac12$. Thus, $E^c$ is not $\alpha$-thin for any $\alpha>0$.
\end{remark}

\begin{lemma}[Upper Bounds for Non-diagonal Terms]\label{lem-91}
	Assume that $E^c$ is $\alpha$-thin with $\alpha\geq \frac{1}{2} $. Then for each $n\in \N^+$,
	\begin{align}\label{equ-911-1}
		\left|\int_E \varphi_j(x) \varphi_k(x) \d x\right|\lesssim \Upsilon_0(\lambda_n,\alpha):=
		\left\{
		\begin{array}{ll}
			\lambda_n^{-\alpha}, \quad & \alpha\in[\frac{1}{2},1),\\
			\lambda_n^{-1}\log\lambda_n, \quad & \alpha=1,\\
			\lambda_n^{-1}, \quad &\alpha>1,
		\end{array}
		\right.
	\end{align}
	when $\lambda_j,\lambda_k\in [\lambda_n-\lambda_n^{\frac{1}{2}},\lambda_n+\lambda_n^{\frac{1}{2}}]$.
\end{lemma}
\begin{proof}
	Notice two facts: First, the orthogonality of $\varphi_k$ and $\varphi_j$ (with $k\neq j$) indicates
	$$
	\int_E\varphi_j\varphi_k\d x = -\int_{E^c}\varphi_j\varphi_k\d x.
	$$
	Second,  Lemma \ref{lem2.1} \noindent $(i)$ shows that $\varphi_j \varphi_k$ is either odd or even.
	Thus, to prove \eqref{equ-911-1}, it suffices to show  that for an arbitrarily fixed   $n\in\N^+$, and $\lambda_j,\lambda_k\in [\lambda_n-\lambda_n^{\frac{1}{2}},\lambda_n+\lambda_n^{\frac{1}{2}}]$,
	\begin{align}\label{equ-91-1}
		\left|\int_{E^c\cap[0,\infty)} \varphi_j(x) \varphi_k(x) \d x\right|\lesssim \Upsilon_0(\lambda_n,\alpha).
	\end{align}
	
	We now show \eqref{equ-91-1}. By  Lemma \ref{lem2.1} \noindent $(ii)$, we have
	\begin{align}\label{equ-91-2}
		|\varphi_j(x)|, |\varphi_k(x)|\lesssim\left\{\begin{array}{lc}
			\lambda_n^{-\frac{1}{4}}  \left(\lambda_n-x\right)^{-\frac{1}{4}}, \,\,\,\,\,\quad\quad\,\,\mbox{if}\,\,\, 0<x<\lambda_n-2\lambda_n^{\frac{1}{2}}, \\
			\lambda_n^{-\frac{1}{4}}, \,\,\qquad\qquad\quad\qquad\,\,\mbox{if}\,\,\, \lambda_n-2\lambda_n^{\frac{1}{2}} \leq x \leq \lambda_n+2\lambda_n^{\frac{1}{2}}, \\
			\lambda_n^{-\frac{1}{4}}   e^{-\frac{2}{3}(x-\lambda_n-\sqrt{\lambda_n})^{\frac{3}{2}}}, \qquad\,\,\mbox{if}\,\,\, x>\lambda_n+2\lambda_n^{\frac{1}{2}}.
		\end{array}\right.
	\end{align}
	Meanwhile, it is clear that
	\begin{align}\label{equ-91-3}
		\left|\int_{E^c\cap[0,\infty)} \varphi_j(x) \varphi_k(x) \d x\right| \leq J_1+J_2+J_3,
	\end{align}
	where
	\begin{align*}
		J_1:=\int_{E^c\cap[0,\lambda_n-2\lambda_n^{\frac{1}{2}})} |\varphi_j(x) \varphi_k(x)| \d x,\qquad J_2:=\int_{E^c\cap[\lambda_n-2\lambda_n^{\frac{1}{2}},\lambda_n+2\lambda_n^{\frac{1}{2}}]} |\varphi_j(x) \varphi_k(x)| \d x,
	\end{align*}
	and
	\begin{equation*}
		J_3:=\int_{E^c\cap(\lambda_n+2\lambda_n^{\frac{1}{2}},\infty)} |\varphi_j(x) \varphi_k(x)| \d x.
	\end{equation*}
	We first deal with $J_1$. Since $E^c$ is $\alpha$-thin, we have
	\begin{align}\label{equ-91-4}
		|E^c\cap[l-1,l]|\lesssim l^{-\alpha}, \;\;\mbox{when}\;\; l\in \N.
	\end{align}
	Then, it follows from \eqref{equ-91-2} and \eqref{equ-91-4} that
	\begin{align}\label{equ-91-5}
		J_1&\lesssim \int_{E^c\cap[0,\lambda_n-2\lambda_n^{\frac{1}{2}})} \lambda_n^{-\frac{1}{2}}(\lambda_n-x)^{-\frac{1}{2}}\d x
		\lesssim \sum_{1\leq l\leq [\lambda_n-2\lambda_n^{\frac{1}{2}}]+1}|E\cap[l-1,l]| \lambda_n^{-\frac{1}{2}}(\lambda_n-l)^{-\frac{1}{2}} \nonumber\\
		&\lesssim \lambda_n^{-\frac{1}{2}}\sum_{1\leq l\leq [\lambda_n-2\lambda_n^{\frac{1}{2}}]+1}l^{-\alpha}(\lambda_n-l)^{-\frac{1}{2}}
		\lesssim \Upsilon_0(\lambda_n,\alpha),
	\end{align}
	where in the last step we used the fact
	\begin{align*}
		&\sum_{1\leq l\leq [\lambda_n-2\lambda_n^{\frac{1}{2}}]+1}l^{-\alpha}(\lambda_n-l)^{-\frac{1}{2}} = \sum_{1\leq l\leq \frac{1}{2}\lambda_n}l^{-\alpha}(\lambda_n-l)^{-\frac{1}{2}}+\sum_{\frac{\lambda_n}{2} <l\leq [\lambda_n-2\lambda_n^{\frac{1}{2}}]+1}l^{-\alpha}(\lambda_n-l)^{-\frac{1}{2}}\\
		& \qquad \lesssim \lambda_n^{-\frac{1}{2}}\sum_{1\leq l\leq \frac{1}{2}\lambda_n}l^{-\alpha} +\lambda_n^{-\alpha}\sum_{\frac{\lambda_n}{2} <l\leq [\lambda_n-2\lambda_n^{\frac{1}{2}}]+1} (\lambda_n-l)^{-\frac{1}{2}}\\
		&\qquad\lesssim \left\{
		\begin{array}{ll}
			\lambda_n^{\frac{1}{2}-\alpha}, \quad & \alpha\in[\frac{1}{2},1),\\
			\lambda_n^{-\frac{1}{2}}\log\lambda_n, \quad & \alpha=1,\\
			\lambda_n^{-\frac{1}{2}}, \quad &\alpha>1.
		\end{array}
		\right.
	\end{align*}
	For $J_2$, we can use a similar way to the above to get
	\begin{align}\label{equ-91-6}
		J_2&\lesssim  \lambda_n^{-\frac{1}{2}}\sum_{ \lambda_n-2\lambda_n^{\frac{1}{2}}\leq l\leq \lambda_n+2\lambda_n^{\frac{1}{2}}} |E^c\cap[l-1,l]|\lesssim \lambda_n^{-\frac{1}{2}}\sum_{\lambda_n-2\lambda_n^{\frac{1}{2}}\leq l\leq \lambda_n+2\lambda_n^{\frac{1}{2}}}l^{-\alpha}\nonumber\\
		&\lesssim \lambda_n^{-\frac{1}{2}-\alpha} \sharp \left\{l\in\N: \lambda_n-2\lambda_n^{\frac{1}{2}}\leq l\leq \lambda_n+2\lambda_n^{\frac{1}{2}}\right\} \lesssim \lambda_n^{-\alpha}.
	\end{align}
	For $J_3$, we use \eqref{equ-91-2} to obtain
	\begin{align}\label{equ-91-7}
		J_3\lesssim  \int_{\lambda_n+2\lambda_n^{\frac{1}{2}}}^\infty \lambda_n^{-\frac{1}{2}}   e^{-\frac{4}{3}(x-\lambda_n-\sqrt{\lambda_n})^{\frac{3}{2}}}\d x\lesssim \lambda_n^{-\frac{1}{2}}e^{-\frac{2}{3}\lambda_n^{\frac{3}{4}}}\lesssim \lambda_n^{-\alpha}.
	\end{align}
	Finally, by \eqref{equ-91-3},  \eqref{equ-91-5}, \eqref{equ-91-6} and \eqref{equ-91-7}, and noting $\lambda_n^{-\alpha}\lesssim \Upsilon_0(\lambda_n,\alpha)$, we are led to \eqref{equ-91-1}. This completes the proof.
\end{proof}

Now we are in a position to prove Proposition \ref{lem-92}.

\noindent \emph{Proof of Proposition \ref{lem-92}.}
We arbitrarily fix $\varepsilon>0$ and write
\begin{align}\label{equ-92-2}
	\mbox{LHS of}\,\, \eqref{equ-92-0}&=\sum_{k\in J_{\alpha}^{\varepsilon}(\lambda_n)}|c_k|^2\int_E{|\varphi_k(x)|^2\,\mathrm dx}+\sum_{k\ne j}c_k\bar{c_j}\int_E{\varphi_k(x)\varphi_j(x)\,\mathrm dx}\nonumber\\
	&:= \uppercase\expandafter{\romannumeral1}+ \uppercase\expandafter{\romannumeral2}.
\end{align}
For the term $\uppercase\expandafter{\romannumeral1}$, we can use  Lemma  \ref{lem-low-h} and Remark \ref{rmk3.1} to see
\begin{align}\label{equ-92-1}
	\uppercase\expandafter{\romannumeral1} \geq D_1\sum_{k\in J_{\alpha}^{\varepsilon}(\lambda_n)}|c_k|^2,
\end{align}
where  $D_1>0$ is given by  Lemma  \ref{lem-low-h}.

To deal with the term $ \uppercase\expandafter{\romannumeral2}$, we first claim the following cardinality bound:
\begin{align}\label{equ-92-3.1}
	\sharp  J_{\alpha}^{\varepsilon}(\lambda_n) \leq
	\left\{
	\begin{array}{ll}
		D_2\varepsilon \lambda_n^{\alpha}, \quad & \alpha\in[\frac{1}{2},1),\\
		D_2\varepsilon \lambda_n(\log\lambda_n)^{-1}, \quad & \alpha=1,\\
		D_2\varepsilon \lambda_n, \quad & \alpha>1,
	\end{array}
	\right.
\end{align}
where $D_2$ is a positive constant, independent of $n$, $\alpha$ and $\varepsilon$.
In fact, when $\alpha\in[\frac{1}{2},1)$, the result in \eqref{equ-92-3.1} is obtained by two facts: first, 
$|\{\lambda\in \R: |\lambda-\lambda_n|\leq  \varepsilon \lambda_n^{\alpha-\frac{1}{2}}\}|=
2 \varepsilon \lambda_n^{\alpha-\frac{1}{2}}$; second, $\lambda_{k+1}-\lambda_k\sim \lambda_n^{-\frac{1}{2}}$,  when $k\in J_{\alpha}^{\varepsilon}(\lambda_n)$. The proof of \eqref{equ-92-3.1} for the case that 
$\alpha\geq 1$ can be done similarly.

Next,   since $E^c$ is $\alpha$-thin, we can apply Lemma \ref{lem-91} and   \eqref{equ-92-3.1}
to find
\begin{align}\label{equ-92-3}
	\uppercase\expandafter{\romannumeral2}
	&\leq \left(\sum_{k\neq j, k,j\in J_{\alpha}^{\varepsilon}(\lambda_n)}|c_jc_k|^2 \right)^{\frac{1}{2}}\left( \sum_{k\neq j, k,j\in J_{\alpha}^{\varepsilon}(\lambda_n)} \left|\int_{E}\varphi_j\varphi_k\d x\right|^2 \right)^{\frac{1}{2}}\nonumber\\
	&\leq \sum_{k\in  J_{\alpha}^{\varepsilon}(\lambda_n)}|c_k|^2\cdot \left( \sum_{k\neq j, k,j\in J_{\alpha}^{\varepsilon}(\lambda_n)} (D_3\Upsilon_0(\lambda_n,\alpha))^2 \right)^{\frac{1}{2}}\nonumber\\
	&\leq  \varepsilon D_2D_3\sum_{k\in  J_{\alpha}^{\varepsilon}(\lambda_n)}|c_k|^2.
\end{align}
where $D_3$ is given by Lemma \ref{lem-91}.

Finally, by plugging \eqref{equ-92-1} and \eqref{equ-92-3} into \eqref{equ-92-2}, and taking $\varepsilon>0$  such that $\varepsilon D_2D_3\leq \frac{D_1}{2}$, we obtain \eqref{equ-92-0}.
This completes the proof of Proposition \ref{lem-92}. \qed

\subsection{Relaxed observability inequality }\label{sec3.2}

For each $k\in \R$, we define the following Sobolev norms associated to $H$:
$$
\|u\|_{\mathcal {H}^k}=\left( \sum_{n=1}^\infty (1+\lambda_n)^{2k}|(u,\varphi_n)|^2 \right)^{\frac{1}{2}},
$$
where $(\cdot,\cdot)$ denotes the inner product in $L^2(\R)$.
The main result of this subsection is the following relaxed observability inequality:

\begin{proposition}\label{lem-ob-comp}
	Let $H$ be given in \eqref{equ0.6}. Suppose that $E^c$ is $\alpha$-thin with $\alpha> \frac{1}{2}$. Then for each $T\in(0,\frac{1}{2})$,
	\begin{align}\label{equ3.23}
		\|u_0\|^2_{L^2(\R)}\lesssim T^{-1} \int_0^T\|e^{-\i t H}u_0\|^2_{L^2(E)}\d t+\Upsilon_1(T,\alpha)\|u_0\|^2_{\mathcal {H}^{-1}}\;\;\mbox{for any}\;\;u_0\in L^2(\R),
	\end{align}
	where
	$$
	\Upsilon_1(T,\alpha)=
	\left\{
	\begin{array}{ll}
		T^{-(1+\frac{6}{2\alpha-1})}, \quad & \alpha\in(\frac{1}{2},1),\\
		T^{-7}(\log \frac{1}{T})^6, \quad &\alpha=1,\\
		T^{-7}, \quad &\alpha>1.
	\end{array}
	\right.
	$$
\end{proposition}
To prove it, we need the following resolvent estimate, which relies on the spectral inequality established in Proposition \ref{lem-92}.
\begin{lemma}\label{lem-92-resolve}
	Let $H$ be given in \eqref{equ0.6}. Suppose that $E^c$ is $\alpha$-thin with $\alpha\geq \frac{1}{2}$. Then for each $\lambda\in \R$,
	\begin{align}\label{equ-92-4}
		\|u\|_{L^2(\R)}\lesssim \Upsilon_2(\lambda,\alpha)\|(H-\lambda)u\|_{L^2(\R)}+\|u\|_{L^2(E)}\;\;\mbox{for all}\;\;u\in D(H),
	\end{align}
	where
	$$
	\Upsilon_2(\lambda,\alpha)=
	\left\{
	\begin{array}{ll}
		(1+|\lambda|)^{-(\alpha-\frac{1}{2})}, \quad & \alpha\in[\frac{1}{2},1),\\
		(1+|\lambda|)^{-\frac{1}{2}}\log(e+|\lambda|), \quad & \alpha=1,\\
		(1+|\lambda|)^{- \frac{1}{2}}, \quad & \alpha>1.\\
	\end{array}
	\right.
	$$
\end{lemma}
\begin{proof}
	First, since $H\geq I$, we have 
	$$
	\|(H-\lambda I)u\|_{L^2(\R)}\geq (1+|\lambda|)\|u\|_{L^2(\R)}\\,\;\;\mbox{when}\;\;\lambda\leq 0,
	$$
	which leads to \eqref{equ-92-4} with $\lambda\leq 0$.
	
	Next, we consider the case $\lambda \gg 1$. For brevity, we only 
	prove  \eqref{equ-92-4} with 
	$\alpha\in[\frac{1}{2},1)$. The proof of \eqref{equ-92-4} with $\alpha\geq 1$ can be done similarly. 
	
	To this end, we arbitrarily fix $\lambda \gg 1$ and $\alpha\in[\frac{1}{2},1)$. Since $E^c$ is $\alpha$-thin, we can use Proposition \ref{lem-92} to find $\varepsilon>0$ and $\delta>0$ such that \eqref{equ-92-0} holds.  
	Let $Q_\lambda$ be the projection onto the eigenspace spanned by the eigenfunctions $\varphi_k$ with
	$|\lambda_k-\lambda|<\varepsilon \lambda^{\alpha-\frac{1}{2}}$. Then we have
	\begin{align}\label{equ-92-9}
		\|(H-\lambda I)(I-Q_{\lambda})u\|_{L^2(\R)}\geq  \varepsilon \lambda^{\alpha-\frac{1}{2}}\|(I-Q_{\lambda})u\|_{L^2(\R)}.
	\end{align}
	Meanwhile, by \eqref{equ-92-0}, we see
	\begin{align}\label{equ-92-10}
		\| Q_{\lambda} u\|_{L^2(E)}\geq  \delta \| Q_{\lambda} u\|_{L^2(\R)}.
	\end{align}
	By \eqref{equ-92-9} and \eqref{equ-92-10}, we find 
	\begin{align*}
		\|u\|_{L^2(\R)}& \leq \|(1-Q_\lambda) u\|_{L^2(\R)}+\|Q_\lambda u\|_{L^2(\R)}
		\lesssim \|(1-Q_\lambda) u\|_{L^2(\R)}+\|Q_\lambda u\|_{L^2(E)}\\
		&\leq 2\|(1-Q_\lambda) u\|_{L^2(\R)}+\|u\|_{L^2(E)}\\
		&\lesssim \lambda^{-(\alpha-\frac{1}{2})}\|(H-\lambda)(1-Q_{\lambda})u\|_{L^2(\R)} +\|u\|_{L^2(E)}\\
		&\leq \lambda^{-(\alpha-\frac{1}{2})}\|(H-\lambda)u\|_{L^2(\R)}+\|   u\|_{L^2(E)}.
	\end{align*}
	In the last inequality above, we used the fact that $\|(1-Q_{\lambda})(H-\lambda)u\|_{L^2(\R)}\le \|(H-\lambda)u\|_{L^2(\R)}$.
	Thus, we have proved \eqref{equ-92-4} for the case $\lambda \gg 1$.

	Finally, we turn to the case where $0<\lambda \lesssim 1$. Then  \eqref{equ-92-4} reduces to
	\begin{align}\label{equ-92-5}
		\|u\|_{L^2(\R)}\lesssim  \|(H-\lambda)u\|_{L^2(\R)}+\|u\|_{L^2(E)}.
	\end{align}
	We assume that $\lambda\leq \Lambda_0$ for some $\Lambda_0>0$. 
	We need the following two observations:
	
	\noindent {\it $(i)$} Since all eigenvalues $\lambda_n$ are simple, there exists a positive constant $\delta$ such that
	$$
	\delta_0:=\inf_{\lambda_j,\lambda_k\leq \Lambda_0}|\lambda_j-\lambda_k|>0.
	$$ 
	{\it ($ii$)} By combining Lemma \ref{lem-low-h}, Remark \ref{rmk3.1} and the assumption that $E^c$ is $\alpha$-thin, we deduce that
	$$
	\|\varphi_k\|_{L^2(\R)}\lesssim \|\varphi_k\|_{L^2(E)}.
	$$

	Based on these two facts, \eqref{equ-92-5} can be proved in a manner analogous to the case $\lambda \gg 1$, provided that we define $Q_\lambda$ as the projection onto the eigenspace spanned by the eigenfunctions $\varphi_k$ with
	$|\lambda_k-\lambda|<\delta_0$.

	
	Hence, we complete the proof.
\end{proof}

Now we are in a position to prove Proposition \ref{lem-ob-comp}.

\noindent \emph{Proof of Proposition \ref{lem-ob-comp}:}
We arbitrarily fix $T\in (0,1/2)$. Let $t\mapsto \chi(t)\in\R$ ($t\in\R$)  be a smooth cut-off function supported over $[0,1]$. 
Let $\chi_T(t):=\chi(t/T)$ ($t\in\R$) and $w(t,x):=\chi_T(t)e^{-\i t H}u_0(x)$ ($(t,x)\in\R\times\R$). Then $w$ satisfies
\begin{align}\label{3.21-12-21-w}
	(i\partial_t-H)w(t,x)=i\chi'_Te^{-\i t H}u_0(x):=v(t,x),\;\;t\in\R, x\in\R.
\end{align}

First, taking the Fourier transform with respect to the variable $t$ in \eqref{3.21-12-21-w}   gives
$$
(-H-\tau)\widehat{w}(\tau,x)=\widehat{v}(\tau,x),\;\;\tau\in\R, x\in\R.
$$
This, along with  Lemma \ref{lem-92-resolve}, leads to
\begin{align}\label{3.22-12-21-w}
	\|\widehat{w}(\tau,\cdot)\|_{L^2(\R_x)}\leq C\Upsilon_2(\tau,\alpha)\|\widehat{v}(\tau,\cdot)\|_{L^2(\R_x)}+C\|\widehat{w}(\tau,\cdot)\|_{L^2(E)},\;\tau\in\R,
\end{align}
where
\begin{align}\label{equ-911-2}
	\Upsilon_2(\tau,\alpha)=
	\left\{
	\begin{array}{ll}
		(1+|\tau|)^{-(\alpha-\frac{1}{2})}, \quad & \alpha\in[\frac{1}{2},1),\\
		(1+|\tau|)^{-\frac{1}{2}}\log(e+|\tau|), \quad & \alpha=1,\\
		(1+|\tau|)^{- \frac{1}{2}}, \quad & \alpha>1.\\
	\end{array}
	\right.
\end{align}
(Here and in what follows, $C$ stands for a positive constant (independent of $t$, $\tau$, $x$ and $\alpha$) which varies in different contents.)  
Taking the $L^2$ norm with respect to the $\tau$ variable  in \eqref{3.22-12-21-w} leads to
\begin{multline}\label{equ-911-2.5}
	\|\widehat{w}\|_{L^2(\R_x\times\R_\tau)} \\
	\leq C\Upsilon_2(N,\alpha)\|\widehat{v}\|_{L^2(\R_x\times\R_\tau)}+C\|\widehat{w}\|_{L^2(E\times \R_\tau)}
	+\left(\int_{|\tau|\leq N}\|\widehat{v}(\tau,\cdot)\|^2_{L^2(\R_x)}\d \tau \right)^{\frac{1}{2}}.
\end{multline}

Second, by the Plancherel theorem, we have
\begin{align}
	\|\widehat{w}\|_{L^2(\R_x\times\R_\tau)}=\|u_0\|_{L^2(\R_x)}\|\chi_T\|_{L^2(\R_t)}\sim T^{\frac{1}{2}}\|u_0\|_{L^2(\R_x)},\label{equ-911-2.6}\\
	\|\widehat{v}(\tau)\|_{L^2(\R_x\times\R_\tau)}=\|u_0\|_{L^2(\R_x)}\|\chi'_T\|_{L^2(\R_t)}\sim T^{-\frac{1}{2}}\|u_0\|_{L^2(\R_x)},\label{equ-911-2.7}
\end{align}
and
\begin{align}\label{equ-911-2.8}
	\|\widehat{w}(\tau)\|_{L^2(E\times \R_\tau)}=\|\chi_Te^{-\i t H}u_0\|_{L^2(E\times \R_t)}\leq \left(\int_0^T\|e^{-\i t H}u_0\|^2_{L^2(E)}\d t \right)^{\frac{1}{2}}.
\end{align}

Third, since $\alpha>\frac12$, we can  choose $N$ large enough such that
\begin{align}\label{equ-911-3}
	C\Upsilon_2(N,\alpha)\|\widehat{v}\|_{L^2(\R_x\times\R_\tau)}\leq \frac{1}{2}\|\widehat{w}\|_{L^2(\R_x\times\R_\tau)}.
\end{align}
Indeed, by \eqref{equ-911-2}, \eqref{equ-911-2.6} and \eqref{equ-911-2.7}, we find that \eqref{equ-911-3} holds, provided that
\begin{align}\label{equ-911-4}
	N\sim
	\left\{
	\begin{array}{ll}
		T^{-(\alpha-\frac{1}{2})^{-1}}, \quad & \alpha\in(\frac{1}{2},1),\\
		T^{-2}(\log\frac{1}{T})^2, \quad & \alpha=1,\\
		T^{-2}, \quad &\alpha>1.
	\end{array}
	\right.
\end{align}

Now, it follows  from \eqref{equ-911-2.5},  \eqref{equ-911-2.8} and \eqref{equ-911-3} that
\begin{align}\label{equ-93-1}
	\|u_0\|_{L^2(\R_x)}\leq C'T^{-\frac{1}{2}}\left(\int_0^T\|e^{-\i t H}u_0\|^2_{L^2(E)}\d t \right)^{\frac{1}{2}}+C'T^{-\frac{1}{2}}\left(\int_{|\tau|\leq N}\|\widehat{v}(\tau,\cdot)\|^2_{L^2(\R_x)}\d \tau \right)^{\frac{1}{2}}.
\end{align}

Next, we deal with the last term in \eqref{equ-93-1}.
For each $n\in\N^+$, we let $u_n=(u_0,\varphi_n)_{L^2(\R_x)}$. Then we have
$$
e^{-\i t H}u_0(x)=\sum_n e^{-\i t\lambda_n}u_n\varphi_n(x),\;t\in\R, x\in\R.
$$
With  \eqref{3.21-12-21-w}, the above yields
$$
\widehat{v}(\tau,x)=\i\sum_n \widehat{\chi'_T}(\tau-\lambda_n)u_n\varphi_n(x),\;\;\tau\in\R, \tau\in\R.
$$
Since $|\widehat{\chi}(\xi)|\lesssim (1+|\xi|)^{-2}$ for each $\xi\in R$, 
we have
$$
|\widehat{\chi'_T}(\xi)|\lesssim |T\xi|(1+T|\xi|)^{-2}\lesssim (1+T|\xi|)^{-1},\;\;\xi\in\R.
$$
This yields  that when  $|\tau|\leq N$,
\begin{align*}
	|\widehat{\chi'_T}(\tau-\lambda_n)|\lesssim  (1+T|\tau-\lambda_n|)^{-1}\lesssim
	\left\{
	\begin{array}{ll}
		T^{-1}(1+\lambda_n)^{-1}, \quad &\lambda_n>2N,\\
		N(1+\lambda_n)^{-1}, \quad &\lambda_n\leq 2N.
	\end{array}
	\right.
\end{align*}
Hence, we have
\begin{align*}
	\left(\int_{|\tau|\leq N}\|\widehat{v}(\tau,\cdot)\|^2_{L^2(\R_x)}\d \tau \right)^{\frac{1}{2}}
	&\lesssim \left(\sum_{n=1}^\infty|u_n|^2\int_{|\tau|\leq N}|\widehat{\chi'_T}(\tau-\lambda_n)|^2\d \tau \right)^{\frac{1}{2}}\\
	&\lesssim \left(\sum_{n=1}^\infty|u_n|^2(1+\lambda_n)^{-2} N(T^{-2}+ N^2) \right)^{\frac{1}{2}}\\
	&\lesssim N^{\frac{1}{2}}(T^{-1}+N)\|u_0\|_{\mathcal {H}^{-1}}\\
	&\lesssim N^{\frac{3}{2}}\|u_0\|_{\mathcal {H}^{-1}},
\end{align*}
where in the last step we used $T^{-1}\lesssim N$, which follows from \eqref{equ-911-4}.  Plugging this into \eqref{equ-93-1} and making use of \eqref{equ-911-4} again, we obtain \eqref{equ3.23}. This completes the proof of Proposition \ref{lem-ob-comp}.
\qed
\begin{remark}\label{rmk3.2}
	In the proof of Proposition \ref{lem-ob-comp}, the following fact is crucial to derive the arbitrariness of $T$ in \eqref{equ3.23}:  $\Upsilon_2(\tau,\alpha)\rightarrow 0$, as $\tau\rightarrow\infty$. The latter requires $\alpha>\frac12$. On the other hand, when $\alpha=\frac12$, it follows from Proposition \ref{lem-92} and Proposition \ref{prop3.1} that $E$ is an observable set at some time  for \eqref{equ1.1}, which is weaker than Theorem \ref{thm1.1}, see Remark \ref{rmk4.1}.
\end{remark}

\subsection{Proof of Theorem \ref{thm1.1}}\label{sec3.3}
In this  subsection, we will prove Theorem \ref{thm1.1} by the relaxed observability inequality \eqref{equ3.23} and the quantitative unique  compactness strategy developed in \cite{Bour13}. This approach enables us not only to effectively glue the high-frequency and low-frequency estimates, but also to obtain an explicit expression for the observable constant $C_{obs}(T,E)$. 

\begin{proof}[Proof of Theorem \ref{thm1.1}]
	We begin by noting that  it suffices to prove the observability inequality \eqref{equ0.2} for any sufficiently small $T>0$. 
	We arbitrarily fix a small  $T\in(0,\frac{1}{2})$ and $u_0\in L^2(\R)$.  We simply write $(\varphi,\psi) $
	for   $(\varphi,\psi)_{L^2(\R}$,
	with $\varphi, \phi\in L^2(\R)$.
	For each $K>0$, we define the spectral projection 
	$$
	\Pi_K \varphi = \sum_{\lambda_k\leq K}(\varphi,\varphi_k)\varphi_k,\;\;\varphi\in L^2(\R).
	$$
	By  Proposition \ref{lem-ob-comp}, we see that for each  $t\in [\frac{T}{4},T]$ and each $K>0$,
	$$
	\|(I-\Pi_K)u_0\|^2_{L^2(\R)}\leq C T^{-1} \int_0^t\|e^{-\i sH}(I-\Pi_K)u_0\|^2_{L^2(E)}\d s+CK^{-2}\Upsilon_1(T,\alpha)\|(I-\Pi_K)u_0\|^2_{L^2},
	$$
	where $\Upsilon_1(T,\alpha)$ is given by  Proposition \ref{lem-ob-comp}.
	So, when $K\geq  (2C\Upsilon_1(T,\alpha))^{\frac{1}{2}}$, 
	\begin{align}\label{equ-93-2}
		\|(I-\Pi_K)u_0\|^2_{L^2(\R)}\leq 2C T^{-1} \int_0^t\|e^{-\i sH}(I-\Pi_K)u_0\|^2_{L^2(E)}\d s, \quad t\in[\frac{T}{4},T].
	\end{align}
	Let
	\begin{align}\label{equ-93-2.5}
		N:=(2C\Upsilon_1(T,\alpha))^{\frac{1}{2}}\sim
		\left\{
		\begin{array}{ll}
			T^{-(\frac{1}{2}+\frac{3}{2\alpha-1})}, \quad & \alpha\in(\frac{1}{2},1),\\
			T^{-\frac{7}{2}}(\log \frac{1}{T})^3, \quad &\alpha=1,\\
			T^{-\frac{7}{2}}, \quad &\alpha>1.
		\end{array}
		\right.
	\end{align}
	Then by \eqref{equ-93-2} (where $t=T/4$ and $u_0$ is replaced by $e^{-\i\frac{T}{2}H}u_0$), we obtain that when $K\geq N$,
	\begin{align}\label{equ-93-3}
		\|(I-\Pi_K)u_0\|^2_{L^2(\R)}\leq 2C T^{-1} \int_{\frac{T}{2}}^{\frac{3T}{4}}\|e^{-\i t H}(I-\Pi_K)u_0\|^2_{L^2(E)}\d t.
	\end{align}
	
	To prove the theorem, we  decompose the initial state $u_0$ into two components
	$$
	u_0=(\Pi_M-\Pi_N)u_0 +(I-\Pi_M+\Pi_N)u_0,
	$$
	where $M$ is chosen to be sufficiently large such that $M>2N$. We  proceed to estimate each term on the right hand side separately and then glue them together. The remainder of the proof is organized into  several steps.

	\noindent{\it Step 1. Estimate  of $\Pi_Nu_0$.}
	
	Since 
	$$
	e^{-\i t H}\Pi_Nu_0=\sum_{\lambda_k\leq N} e^{-\i t\lambda_k}u_k\varphi_k,\;\;\mbox{with}\;\;u_k=(u_0,\varphi_k),
	$$
	we can use Fubini's theorem to get
	\begin{align}\label{equ-93-4}
		\int_{\frac{T}{2}}^{\frac{3T}{4}}\|e^{-\i t H}\Pi_Nu_0\|^2_{L^2(E)}\d t=\int_E\int_{\frac{T}{2}}^{\frac{3T}{4}}\Big|\sum_{\lambda_k\leq N} e^{-\i t\lambda_k}u_k\varphi_k\Big|^2\d t\d x.
	\end{align}
	Next we give a lower bound for the term on the right hand of \eqref{equ-93-4}. To this end,
	we use  \eqref{equ2.4} to obtain the following uniform gap for $\{\lambda_k\}_{\lambda_k\leq N}$:
	\begin{align}\label{equ-93-5}
		\Delta \sim N^{-\frac{1}{2}}.
	\end{align}
	Thus, we can apply Lemma \ref{lem-Salem}, i.e., the Salem inequality,
	to obtain
	\begin{align}\label{equ-93-6}
		\sum_{\lambda_k\leq N}|u_k\varphi_k(x)|^2\leq \frac{4}{|I|} \int_{I}\big|\sum_{\lambda_k\leq N} e^{-\i t\lambda_k}u_k\varphi_k(x)\big|^2\d t,\;\;x\in\R,
	\end{align}
	provided that the interval $I$ satisfies $|I|\geq 4\pi/\Delta$, which is possible if
	\begin{align}\label{equ-93-7}
		|I|\sim N^{\frac{1}{2}}.
	\end{align}
	Without loss of generality, we can assume that $[0,T]\subset I$. Let
	$$
	n_0=\sharp\{\lambda_k: \lambda_k\leq N\}.
	$$
	Then if follows from  \eqref{equ2.3} that $n_0 \leq N^{\frac{3}{2}}$ for $T>0$ sufficient small. 
	Applying Lemma \ref{lem-Nazarov} (the Nazarov inequality),
	we obtain that
	\begin{align}\label{equ-93-8}
		\int_{I}\Big|\sum_{\lambda_k\leq N} e^{-\i t\lambda_k}u_k\varphi_k(x)\Big|^2\d t \leq \left( \frac{4A|I|}{T}\right)^{2N^{\frac{3}{2}}-1}\int_{\frac{T}{2}}^{\frac{3T}{4}}\Big|\sum_{\lambda_k\leq N} e^{-\i t\lambda_k}u_k\varphi_k(x)\Big|^2\d t,\quad x\in \R,
	\end{align}
	where $A>0$ is the constant given by Lemma \ref{lem-Nazarov}.
	
	Combining \eqref{equ-93-6} and \eqref{equ-93-8}, and integrating over $E$ we have
	\begin{align}\label{equ-93-9}
		\frac{4}{|I|}\left( \frac{4A|I|}{T}\right)^{2N^{\frac{3}{2}}-1}\int_E\int_{\frac{T}{2}}^{\frac{3T}{4}}\Big|\sum_{\lambda_k\leq N} e^{-\i t\lambda_k}u_k\varphi_k(x)\Big|^2\d t\d x\geq  \int_E\sum_{\lambda_k\leq N}|u_k\varphi_k(x)|^2\d x.
	\end{align}
	Meanwhile, by Lemma  \ref{lem-low-h} and Remark \ref{rmk3.1}, we have  that for some $c>0$ (independent of $k$), 
	$$\inf_k\int_E|\varphi_k(x)|^2\d x \geq c,$$
	which yields
	\begin{align}\label{equ-93-10}
		\int_E\sum_{\lambda_k\leq N}|u_k\varphi_k(x)|^2\d x\geq c\sum_{\lambda_k\leq N}|u_k|^2=c\big\|\Pi_Nu_0\big\|^2_{L^2(\R)}.
	\end{align}
	Finally, it follows from \eqref{equ-93-4}, \eqref{equ-93-9} and \eqref{equ-93-10} that for some absolute constant $A_1>0$
	\begin{align}\label{equ-93-11}
		\big \|\Pi_Nu_0\big\|^2\leq  \left( \frac{A_1N^{\frac{1}{2}}}{T}\right)^{2N^{\frac{3}{2}}}\int_{\frac{T}{2}}^{\frac{3T}{4}}\big\|e^{-\i t H}\Pi_Nu_0\big\|^2_{L^2(E)}\d t.
	\end{align}

	\noindent{\it Step 2. Estimate  of $(I-\Pi_M+\Pi_N)u_0$.}
	
	Let $M>2N$ be a quantity to be determined later. 
	Let $\eta\in C_0^\infty(0,1)$ be a cutoff function such that $\eta=1$ on $[\frac{1}{2},\frac{3}{4}]$. 
	Denote $\eta_T(t)=\eta(t/T)$ for  $t\in\R$ and
	write $\mathbf{1}_E$ for the characteristic function of $E$. Then we have
	\begin{multline}\label{equ-94-1}
		\int_\R \|e^{-\i t H}(I-\Pi_M+\Pi_N)u_0\|^2_{L^2(E)}\eta_T(t)\d t  = \int_\R \|e^{-\i t H}(I-\Pi_M)u_0\|^2_{L^2(E)}\eta_T(t)\d t\\
		+\int_\R \|e^{-\i t H} \Pi_Nu_0\|^2_{L^2(E)}\eta_T(t)\d t+ \mathcal{I},
	\end{multline}
	where
	\begin{multline}\label{equ-94-2}
		\mathcal{I}
		:=
		2 \mbox{ Re} \int_\R(\mathbf{1}_Ee^{-\i t H}(I-\Pi_M)u_0,\mathbf{1}_Ee^{-\i t H}\Pi_Nu_0)\eta_T(t)\d t\\
		= 2 \mbox{ Re} \sum_{\lambda_n\leq N}\sum_{\lambda_m>M} u_nu_m(\mathbf{1}_E\varphi_n,\mathbf{1}_E \varphi_m)\int_\R e^{\i(\lambda_n-\lambda_m)t} \eta_T(t)\d t.
	\end{multline}
	Notice that for each $l\in\N^+$, there is an absolute constant $D_l>0$ such that
	\begin{align}\label{3.44-12-23-w}
		|\widehat{\eta_T}(\tau)|\leq D_lT(1+|T\tau|)^{-l},\;\;\tau\in\R.
	\end{align}
	By 
	\eqref{equ-94-2} and \eqref{3.44-12-23-w}, we see that  for each $l\in\N^+$,
	\begin{align}\label{equ-94-3}
		|\mathcal{I}|&\leq  2\sum_{\lambda_n\leq N}\sum_{\lambda_m>M}|u_nu_m||
		\widehat{\eta_T}(\lambda_m-\lambda_n)|\nonumber\\
		&\leq \sum_{\lambda_n\leq N}\sum_{\lambda_m>M}(|u_n|^2+|u_m|^2)
		D_lT(1+T(\lambda_m-\lambda_n))^{-l}\nonumber\\
		\mbox{ (since $M>2N$) }&\leq \sum_{\lambda_n\leq N}\sum_{\lambda_m>M}(|u_n|^2+|u_m|^2)
		D_lT(1+2^{-1}T \lambda_m )^{-l}\nonumber\\
		&\leq I_{1,l}+I_{2,l},
	\end{align}
	where
	$$
	I_{1,l}:=\sum_{\lambda_m>M} D_lT(1+2^{-1}T \lambda_m )^{-l}\|u_0\|^2_{L^2(\R)},
	$$
	and
	\begin{equation*}
		I_{2,l}:= \sharp\{n:\lambda_n\leq N\} \sum_{\lambda_m>M} |u_m|^2 D_lT(1+2^{-1}T \lambda_m )^{-l}.
	\end{equation*}
	For the term $I_{1,l}$, using the asymptotic relation $\lambda_m\sim m^{\frac{2}{3}}$, we see that for $l\geq 2$,
	\begin{align}\label{equ-94-4}
		I_{1,l}\lesssim C_l2^lT^{1-l}\sum_{\lambda_m>M}m^{-\frac{2}{3}l}\|u_0\|^2_{L^2(\R)}\lesssim D_l2^lT^{1-l}M^{\frac{3}{2}-l}\|u_0\|^2_{L^2(\R)}.
	\end{align}
	For  the term  $I_{2,l}$, noting that $\sharp\{n:\lambda_n\leq N\}\lesssim N^{\frac{3}{2}}\lesssim M^{\frac{3}{2}}$, we have
	that for each $l\in\N^+$,
	\begin{align}\label{equ-94-5}
		I_{2,l}\lesssim M^{\frac{3}{2}}\sum_{\lambda_m>M} |u_m|^2 D_lT(1+2^{-1}T M )^{-l}\lesssim D_l2^lT^{1-l}M^{\frac{3}{2}-l}\|u_0\|^2_{L^2(\R)}.
	\end{align}
	Combining \eqref{equ-94-3}-\eqref{equ-94-5} and setting $l=2$, we deduce  that there is an absolute constant $C_1>0$ such that
	\begin{align}\label{equ-94-6}
		|\mathcal{I}| \leq C_1T^{-1}M^{-\frac{1}{2}}\|u_0\|^2_{L^2(\R)}.
	\end{align}
	Putting \eqref{equ-93-3} (whith $K=N$), \eqref{equ-93-11} and \eqref{equ-94-6} into \eqref{equ-94-1}, and recalling that $\eta_T=1$ on $[\frac{T}{2},\frac{3T}{4}]$, we find
	\begin{multline}\label{equ-94-7}
		\int_{0}^{T} \big\|e^{-\i t H}(I-\Pi_M+\Pi_N)u_0\big\|^2_{L^2(E)} \d t \geq (2C)^{-1}T\|(I-\Pi_M)u_0\|^2_{L^2(\R)}\\
		+ \left( \frac{A_1N^{\frac{1}{2}}}{T}\right)^{-2N^{\frac{3}{2}}}\|\Pi_Nu_0\|^2_{L^2(\R)}-C_1T^{-1}M^{-\frac{1}{2}}\|u_0\|^2_{L^2(\R)}.
	\end{multline}
	
	\noindent 
	{\it Step 3. Estimate  of $(\Pi_M-\Pi_N)u_0$.} 
	
	Let $\lambda_n$ (resp. $\lambda_m$) be the largest eigenvalue that is less than or equal to $N$ (resp. $M$). 
	Then, we have
	\begin{align}\label{3.50-12-23-w}
		\Pi_N(L^2(\R))=\mbox{span}\{\varphi_1,\varphi_2,\cdots,\varphi_n\};\;\;\;\Pi_M(L^2(\R))=\mbox{span}\{\varphi_1,\varphi_2,\cdots,\varphi_m\}.
	\end{align}
	We define 
	\begin{align}\label{equ-94-8}
		\tau=\frac{T}{10m}
	\end{align}
	and  introduce the  Vandermonde matrix:
	\begin{align}\label{equ-94-9}
		\mathbb{A}=(e^{-\i \lambda_jk\tau})_{1\leq j, k\leq m}.
	\end{align}
	We claim that there is a numerical constant $A_2>0$ such that
	\begin{align}\label{equ-94-10}
		\|\mathbb{A}^{-1}\|\leq (\frac{A_2m}{T})^{m^2}.
	\end{align}
	The proof of \eqref{equ-94-10} is given in Appendix \ref{app-002}.

	Since $\mathbb{A}$ is non-singular, for every $b\in \R^m$, there is a unique solution $y\in \R^m$ such that $\mathbb{A}y=b$. 
	We take  
	$$
	b=(\underbrace{0,\cdots,0}_{n },\underbrace{1,\cdots, 1}_{m-n})^T.
	$$
	Then the corresponding  solution $y=(y_1,\cdots,y_m)^T$ satisfies that
	\begin{align}\label{equ-94-11}
		\sum_{1\leq k\leq m}y_ke^{-\i \lambda_jk\tau}=0\;\;\mbox{ for }\;\; j\leq n;\;\; \sum_{1\leq k\leq m}y_ke^{-\i \lambda_jk\tau}=1\;\;
		\mbox{for}\;\; n<j\leq m.
	\end{align}
	With the above  $y$, we define a modified initial datum:
	\begin{align}\label{equ-94-12}
		\widetilde{u_0}=\sum_{n\geq 1}\Big(\sum_{1\leq k\leq m} y_ke^{-\i \lambda_nk\tau}\Big)u_n\varphi_n,\;\;\mbox{with}\;\; u_n=(u_0,\varphi_n).
	\end{align}
	It follows from \eqref{equ-94-11}, \eqref{equ-94-12} and \eqref{3.50-12-23-w} that
	$$
	(1-\Pi_N)\widetilde{u_0}=\widetilde{u_0}.
	$$
	This, along with 
	\eqref{equ-93-3} (where $u_0=\widetilde{u_0}$ and $K=N$), yields
	\begin{multline}\label{equ-94-13}
		\frac{2C}{T}\int_{\frac{T}{2}}^{\frac{3T}{4}}\|e^{-\i t H}(I-\Pi_N)\widetilde{u_0}\|^2_{L^2(E)}\d t\geq \|\widetilde{u_0}\|^2_{L^2(\R)} =  \sum_{\lambda_n>N}\big|\sum_{1\leq k\leq m} y_ke^{-\i \lambda_nk\tau}\big|^2\,|u_n|^2\\
		\ge \sum_{N<\lambda_n\leq M}\big|\sum_{1\leq k\leq m} y_ke^{-\i \lambda_nk\tau}\big|^2|u_n|^2
		= \|(\Pi_M-\Pi_N)u_0\|^2_{L^2(\R)},
	\end{multline}
	where in the last step we used the second identity in \eqref{equ-94-11}.

	Meanwhile, by \eqref{equ-94-12}, we also have 
	$$
	e^{-\i t H}\widetilde{u_0}=\sum_{1\leq k\leq m}y_ke^{-\i(t+k\tau)H}u_0.
	$$
	Since $y=\mathbb{A}^{-1}b$, the above, together with \eqref{equ-94-8}
	and \eqref{equ-94-10}, shows 
	\begin{align}\label{equ-94-14}
		\int_{\frac{T}{2}}^{\frac{3T}{4}}\|e^{-\i t H}(I-\Pi_N)\widetilde{u_0}\|^2_{L^2(E)}\d t
		&\leq m \max_{1\leq k\leq m}|y_k| \int_{0}^{T}\|e^{-\i t H}u_0\|^2_{L^2(E)}\d t\nonumber\\
		&\leq m^2\Big(\frac{A_2m}{T}\Big)^{m^2}\int_{0}^{T}\|e^{-\i t H}u_0\|^2_{L^2(E)}\d t.
	\end{align}
	Combining \eqref{equ-94-13} and \eqref{equ-94-14} gives
	\begin{align}\label{equ-94-15}
		\frac{2C}{T}m^2\Big(\frac{A_2m}{T}\Big)^{m^2}\int_{0}^{T}\|e^{-\i t H}u_0\|^2_{L^2(E)}\d t\geq \|(\Pi_M-\Pi_N)u_0\|^2_{L^2(\R)}.
	\end{align}

	\noindent {\it Step 4. Gluing together.} 
	
	Using the triangle inequality and \eqref{equ-94-15}, we infer that
	\begin{align}\label{equ-94-16}
		\int_{0}^{T} \|e^{-\i t H}(I-\Pi_M+\Pi_N)u_0\|^2_{L^2(E)} \d t &\leq 2\int_{0}^{T} \|e^{-\i t H}u_0\|^2_{L^2(E)} \d t+2 \int_{0}^{T} \|e^{-\i t H}(\Pi_M-\Pi_N)u_0\|^2_{L^2(E)} \d t\nonumber\\
		&\leq   2\int_{0}^{T} \|e^{-\i t H}u_0\|^2_{L^2(E)} \d t+2 \|(\Pi_M-\Pi_N)u_0\|^2_{L^2(\R)}\nonumber\\
		&\leq \Big(2+\frac{4C}{T}m^2\Big(\frac{A_2m}{T}\Big)^{m^2}\Big)\int_{0}^{T}\|e^{-\i t H}u_0\|^2_{L^2(E)}\d t.
	\end{align}
	Combining \eqref{equ-94-7} and \eqref{equ-94-15}-\eqref{equ-94-16} gives
	\begin{multline}\label{equ-94-17}
		\Big(2+\frac{6C}{T}m^2\Big(\frac{A_2m}{T}\Big)\Big)\int_{0}^{T}\|e^{-\i t H}u_0\|^2_{L^2(E)}\d t
		\geq (2C)^{-1}T\|(I-\Pi_M)u_0\|^2_{L^2(\R)}\\ + \left( \frac{A_1N^{\frac{1}{2}}}{T}\right)^{-2N^{\frac{3}{2}}}\|\Pi_Nu_0\|^2+\|(\Pi_M-\Pi_N)u_0\|^2_{L^2(\R)}-C_1T^{-1}M^{-\frac{1}{2}}\|u_0\|^2_{L^2(\R)}.
	\end{multline}
	Choosing $M>2N$ large enough so that
	\begin{align}\label{equ-94-18}
		C_1T^{-1}M^{-\frac{1}{2}}\leq \frac{1}{2}\min\left\{(2C)^{-1}T, \left( \frac{A_1N^{\frac{1}{2}}}{T}\right)^{-2N^{\frac{3}{2}}} ,1 \right\}.
	\end{align}
	Then by \eqref{equ-94-17}, we obtain
	\begin{align}\label{equ-94-19}
		\|u_0\|^2_{L^2(\R)}\leq C_1^{-1}TM^{\frac{1}{2}}(2+\frac{6C}{T}m^2(\frac{A_2m}{T})^{m^2})\int_{0}^{T}\|e^{-\i t H}u_0\|^2_{L^2(E)}\d t.
	\end{align}

	Finally, we  express  $C_{obs}$ in terms of $T$. 
	In the case $\alpha\in(\frac{1}{2},1)$, we see from  \eqref{equ-93-2.5} that
	$M$ satisfies   \eqref{equ-94-18}, provided that 
	\begin{align}\label{equ-94-20}
		M=\left(C_2N \right)^{-2N^{\frac{3}{2}}}\leq \exp\left\{C_3T^{- \frac{3}{2}(\frac{1}{2}+\frac{3}{2\alpha-1})}\log \frac{1}{T}\right\},
	\end{align}
	where $C_2>0$ and $C_3>0$ are absolute  constants. 
	Meanwhile, since $m^{\frac{2}{3}}\sim \lambda_m\sim M$, we find $m\sim M^{\frac{3}{2}}$. 
	This, together with  \eqref{equ-94-20} and \eqref{equ-94-19}, gives
	$$
	\|u_0\|^2_{L^2(\R)}\leq C_{obs}\int_{0}^{T}\|e^{-\i t H}u_0\|^2_{L^2(E)}\d t
	$$
	with
	\begin{align*}
		C_{obs}\leq m^{\mathcal {O}(m^2)}\leq M^{\mathcal {O}(M^3)} \leq e^{\mathcal {O}(M^4)}\leq e^{C_4e^{4C_3T^{- \frac{3}{2}(\frac{1}{2}+\frac{3}{2\alpha-1})}\log \frac{1}{T}}}\leq e^{ e^{C_5T^{- \frac{3}{2}(\frac{1}{2}+\frac{3}{2\alpha-1})}\log \frac{1}{T}}},
	\end{align*}
	where $C_4>0$ and $C_5>0$ are absolute  constants. 
	In the case $\alpha\geq 1$,  similar computations show that for some absolute constant $C_6>0$,
	$$
	C_{obs}\leq
	\left\{
	\begin{array}{ll}
		\exp\left\{\exp\big[C_6T^{- \frac{21}{4}}(\log \frac{1}{T})^{\frac{11}{2}}\big]\right\}, \quad & \alpha=1,\\
		\exp\left\{\exp\big[ C_6T^{- \frac{21}{4}}\log \frac{1}{T}\big]\right\}, \quad & \alpha>1.
	\end{array}
	\right.
	$$
	This completes the proof.
\end{proof}

\begin{remark}\label{rmk4.2}
	Theorem \ref{thm1.1} shows that any set $E$ with $E^c$ being $\alpha$-thin ($\alpha>1/2$) qualifies as an observable set at any time for \eqref{equ1.1}. While \cite{SSY} presents that the thickness is the characterization for the observable sets for the free  Schr\"{o}dinger equations. However, 
	it follows from 
	Remark \ref{rmk3.1} that the class of sets $E$ with $E^c$ $\alpha$-thin ($\alpha>1/2$) is
	strictly contained in the  broader class of thick sets. 
	It is interesting to explore that whether thick sets are also  observable sets of equation \eqref{equ1.1}.
\end{remark}

\subsection{Perturbation with a bounded potential}\label{app-per}

This subsection shows that Theorem \ref{thm1.1} remains true (up to the explicit control cost $C_{obs}$) for the operator $-\frac{\d^2}{\d x^2}+|x|+V(x)$, where $V\in L^\infty(\R)$. The key of its proof is the following resolvent estimate:
\begin{lemma}\label{lem-pert-1}
	Let $H=-\frac{\d^2}{\d x^2}+|x|$. Let $V\in L^\infty(\R)$ be real valued. Suppose that $E^c$
	is $\alpha$-thin with $\alpha> \frac{1}{2}$. Then for each $\lambda\in \R$,
	\begin{align}\label{equ-125-1}
		\|u\|_{L^2(\R)}\lesssim \Upsilon_2(\lambda,\alpha)\|(H+V-\lambda)u\|_{L^2(\R)}+\|u\|_{L^2(E)},
		\;\;u\in D(H),
	\end{align}
	where $\Upsilon_2(\lambda,\alpha)$ is given by Lemma \ref{lem-92-resolve}.
\end{lemma}
\begin{proof}
	It follows from Lemma \ref{lem-92-resolve} that  for all $\lambda\in \R$ and $u\in D(H)$,
	\begin{align}\label{equ-125-2}
		\|u\|_{L^2(\R)}&\lesssim \Upsilon_2(\lambda,\alpha)\|(H-\lambda)u\|_{L^2(\R)}+\|u\|_{L^2(E)}\nonumber\\
		&\lesssim \Upsilon_2(\lambda,\alpha)\|(H+V-\lambda)u\|_{L^2(\R)}+\Upsilon_2(\lambda,\alpha)\|V\|_{L^\infty(\R)}\|u\|_{L^2(\R)}+\|u\|_{L^2(E)}.
	\end{align}
	Since $V\in L^\infty(\R)$ and $\Upsilon_2(\lambda,\alpha)\to 0$ as $\lambda\to \infty$, we can choose $\lambda_0>0$ large enough so that, for all $\lambda\geq \lambda_0$, $\Upsilon_2(\lambda,\alpha)\|V\|_{L^\infty(\R)}\|u\|_{L^2(\R)}$ can be absorbed by
	the term on the left hand side of \eqref{equ-125-2}. This proves \eqref{equ-125-1} for $\lambda\geq \lambda_0$.
	
	Next, we suppose that   $\lambda \leq -1-2\|V\|_{L^\infty(\R)}$. Then 
	$$
	((H+V-\lambda)u,u)\geq (|\lambda|-\|V\|_{L^\infty(\R)})\|u\|^2_{L^2(\R)}\gtrsim (1+|\lambda|)\|u\|^2_{L^2(\R)},\; u\in D(H).
	$$
	This, together with the Cauchy-Schwartz inequality, implies that
	$$
	\|u\|_{L^2(\R)}\lesssim (1+|\lambda|)^{-1}\|(H+V-\lambda)u\|_{L^2(\R)}\lesssim \Upsilon_2(\lambda,\alpha)\|(H-\lambda)u\|_{L^2(\R)},\;u\in D(H).
	$$
	Thus \eqref{equ-125-1} also holds in this case.
	
	Finally, we suppose that  $-1-2\|V\|_{L^\infty(\R)}<\lambda<\lambda_0$. Then, \eqref{equ-125-1} is equivalent to 
	\begin{align}\label{equ-125-3}
		\|u\|_{L^2(\R)}\lesssim  \|(H+V-\lambda)u\|_{L^2(\R)}+\|u\|_{L^2(E)}.
	\end{align}
	By contradiction, we 
	suppose that  \eqref{equ-125-3}
	is not true. Then there exist sequences $\{u_n\}_{n\geq 1}\subset D(H)$ and $\{\lambda_n\}_{n\geq 1}\subset (-1-2\|V\|_{L^\infty(\R)},\lambda_0)$ such that
	\begin{align}\label{equ-125-4}
		\|u_n\|_{L^2(\R)}=1, \quad  \|(H+V-\lambda_n)u_n\|_{L^2(\R)}+\|u_n\|_{L^2(E)}\leq \frac{1}{n}\; \; \mbox{for all}\;\; n\geq 1.
	\end{align}
	This yields that $\{u_n\}$ is uniformly bounded in $D(H)$. Thus, without loss of generality, we can assume that $u_n\to u$ weakly in $D(H)$, $\lambda_n\to \lambda \in [-1-2\|V\|_{L^\infty(\R)},\lambda_0]$ as $n\to \infty$. Since the embedding $D(H)\hookrightarrow L^2(\R)$ is compact, we find $u_n \to u$ strongly in $L^2(\R)$. Thus
	\begin{align}\label{equ-125-5}
		(H+V-\lambda)u=0 \,\,\mbox{ on }\, \R, \quad u=0\,\, \mbox{ a.e. on } \,\, E.    
	\end{align}
	Since $u\in D(H)$, we find that $u'$ is continuous and $u(x)=0$ for each $x\in E$.
	
	We claim that $u'=0$ almost everywhere on $E$. For this purpose,
	it suffices 
	to show that $u'(x_0)=0$ for an arbitrarily fixed point $x_0$ of density of $E$. Clearly, $x_0$ satisfies
	\begin{align}\label{equ-125-5.5}
		\lim_{\varepsilon\to 0}\frac{|E\cap (x_0-\varepsilon,x_0+\varepsilon)|}{|(x_0-\varepsilon,x_0+\varepsilon)|}=1. 
	\end{align}
	From \eqref{equ-125-5.5}, we can find a sequence $\{x_n\}\subset E$ such that $x_n\to x_0$. Thus, $u(x_0)=0$ and $u(x_n)=0$ for all $n\geq 1$. Consequently, 
	$$
	u'(x_0)=\lim_{n\to \infty}\frac{u(x_n)-u(x_0)}{x_n-x_0}=0,
	$$
	which leads to the claim above. 
	
	Now we arbitrarily  fix  $x_0\in E$ such that $u(x_0)=u'(x_0)=0$. It follows from \eqref{equ-125-5} that
	\begin{align}\label{equ-125-6}
		u(x)=\int^x_{x_0}\int^y_{x_0} (|z|+V(z)-\lambda)u(z)\d z \d y, \quad x\in \R. 
	\end{align}
	Let $I$ be a compact interval containing $x_0$. We define a linear operator $T:L^\infty(I)\to L^\infty(I)$ in the manner: for each $v\in L^\infty(I)$, we let $Tv$ to be  the right hand side of \eqref{equ-125-6} where $u=v$. Then, it is easy to see that $T^m$ is a contraction mapping on $L^\infty(I)$ for $m$ large enough. (Note that $m$ depends  only on $\lambda, \|V\|_{L^\infty(\R)}$ and $|I|$.) Thus \eqref{equ-125-6} has a unique solution, that is, $u\equiv 0$ on $I$. Since the size $|I|$ can be arbitrarily large, we find $u\equiv 0$ on $\R$. This  contradicts with $\|u\|_{L^2(\R)}=1$ in \eqref{equ-125-4}. Thus \eqref{equ-125-3} is proved.
	
	This completes the proof.
\end{proof}

With Lemma \ref{lem-pert-1} in hand, we can use the similar to that proving Proposition \ref{lem-ob-comp} to obtain that for any $T\in(0,\frac{1}{2})$,
\begin{align}\label{equ-125-10}
	\|u_0\|^2_{L^2(\R)}\lesssim T^{-1} \int_0^T\|e^{-\i t(H+V)}u_0\|^2_{L^2(E)}\d t+\Upsilon_1(T,\alpha)\|u_0\|^2_{\widetilde{\mathcal {H}}^{-1}},\;\;u_0\in L^2(\R),
\end{align}
where $\Upsilon_1(T,\alpha)$ is given by Proposition \ref{lem-ob-comp}, and 
$$
\|u\|_{\widetilde{\mathcal {H}}^k}:=\left( \sum_{n=1}^\infty (1+\lambda_n)^{2k}|(u,\varphi_n)|^2 \right)^{\frac{1}{2}},\;\;k\in \R
$$
with $\lambda_n$ and $\varphi_n$ being eigenvalues and eigenfunctions of $H+V$ in $L^2(\R)$, respectively.

Thanks to Lemma \ref{lem-pert-1} and \eqref{equ-125-10}, we can apply the general result in \cite[Theorem 4]{Bour13} to obtain the following observability inequality:
\begin{theorem}\label{thm-potent}
	Let $H=-\frac{\d^2}{\d x^2}+|x|$  and $V\in L^\infty(\R)$ be real-valued. Assume that $E\subset \R$ is a measurable set such that its complement set $E^c$ is $\alpha$-thin  with $\alpha>\frac12$. Then for any $T>0$, there exists  $C_{obs}$ depending only on $\alpha,E,T$ and $\|V\|_{L^\infty(\R)}$ such that
	$$
	\|u_0\|^2_{L^2(\R)}\leq C_{obs}\int_0^T\|e^{-\i t(H+V)}u_0\|^2_{L^2(E)}\d t, \;\; \mbox{ for all } u_0\in L^2(\R).
	$$
\end{theorem}

\section{Observability at some time and the proof of Theorem \ref{thm1.2}}\label{sec4}
We first present a characterization of \emph{an observable set at some time} for equation \eqref{equ1.1}.
Given $\varepsilon>0$ and $n\in\N^+$, we simply write  $J^\varepsilon(\lambda_n)$ 
for the index set $J^\varepsilon_{1/2}(\lambda_n)$ defined by \eqref{equ-92-1.1}.
Then for each $n\in\N^+$,  there is $n_0, K_n\in\N$ such that 
\begin{align}\label{equ3.1}
	J^\varepsilon(\lambda_n):= \{k\in \mathbb{N}^+ \;:\; |\lambda_k-\lambda_n|<\varepsilon\}
	=\{n_0+1,\cdots, n_0+K_n\}.
\end{align}
\begin{proposition}\label{prop3.1}
	Let $E\subset \R$ be a measurable set. Let $\{\varphi_k\}_{k=1}^\infty$ be the sequence of eigenfunctions of $H=-\frac{\d^2}{\d x^2}+|x|$ (normalized in $L^2$) associated with the eigenvalues $\{\lambda_k\}_{k=1}^\infty$.
	Then following  statements are equivalent:
	
	\noindent $(i)$ The set $E$ is  an observable set  at some time  for the equation \eqref{equ1.1}.
	
	\noindent $(ii)$  There exists  $\varepsilon>0$ and $\delta>0$ such that for all $n\ge 1$, the sequence of positive semi-definite matrices $\{A_{E, n}\}_{n=1}^{+\infty}$, where
	\begin{align}\label{equ3.3}
		A_{E, n}=\left(a_{k,j}\right)_{K_n\times K_n},\;\; \mbox{where}\;\; a_{k,j}=\int_{E}{\varphi_{k}(x)\cdot\varphi_{j}(x)\,\mathrm dx},\;\; k,j\in J^\varepsilon(\lambda_n)
	\end{align}
	are uniformly positive in the sense that the smallest eigenvalue of $A_{E, n}$ is uniformly bounded below, i.e., there exists  $\delta>0$ independent of $n$ such that
	\begin{equation}\label{equ3.4}
		\lambda_{1}(A_{E, n})\ge \delta>0,\;\;\,\,\, \mbox{for all}\;\; n\in\mathbb{N}^+.
	\end{equation}
\end{proposition}
\begin{proof}
	By \cite[Proposition 2.1]{Ga},  the resolvent of $H=-\frac{\d^2}{\d x^2}+|x|$  is compact. Then  by \cite[Theorem 1.3]{RTTT},  the conclusion $(i)$ holds if and only if
	there exists  $\varepsilon>0$ and $\delta>0$ such that for each $n\ge 1$,
	\begin{align}\label{equ3.2}
		\int_{E}\left|\sum_{k\in J^\varepsilon(\lambda_n)}c_k\varphi_k(x)\right|^2\d x \geq \delta\cdot\sum_{k\in J^\varepsilon(\lambda_n)}|c_k|^2\;\;\mbox{for all}\;\; c_k\in \mathbb{C}.
	\end{align}
	We arbitrarily fix each $n\in\N^+$. Since
	\begin{equation}\label{equ3.4.1}
		0\leq \int_{E}\left|\sum_{k\in J^\varepsilon(\lambda_n)}c_k\varphi_k(x)\right|^2\d x=X^{H}A_{E, n}X
		\;\;\mbox{for each}\;\;X=(c_1,\cdots,c_{K_n})^{T}\in \mathbb{C}^{K_n},
	\end{equation}
	$A_{E, n}$  is positive semi-definite. 
	Moreover, it follows from the Rayleigh theorem that 
	$$
	\min_{X\ne 0}\frac{X^{H}A_{E, n}X}{X^{H}X}=\lambda_{1}(A_{E, n}).
	$$
	Therefore \eqref{equ3.2} holds if and only if \eqref{equ3.4} holds for some $\delta>0$ independent of $n$. This shows that \noindent $(i)$ and \noindent $(ii)$ are equivalent and completes the proof.
\end{proof}
\begin{remark}\label{rmk4.1}
	Inequality \eqref{equ3.2} coincides with \eqref{equ-92-0} for $\alpha=\frac12$.  Therefore, by Proposition \ref{lem-92} (with $\alpha=\frac12$) and Proposition \ref{prop3.1}, we see that if $E^c$ is $\alpha$-thin with $\alpha= \frac12$, then  $E$  is an \emph{observable set} at some time for \eqref{equ1.1}.
\end{remark}

\subsection{Proof Theorem \ref{thm1.2}}\label{sec4.1}
We first prove the conclusion  \noindent $(i)$. Suppose that $E\subset \mathbb{R}$  is \emph{an observable set at some time} for \eqref{equ1.1}. Then it follows from Proposition \ref{prop3.1} that there exists an absolute constant $\delta>0$ such that $\int_E|\varphi_k(x)|^2\d x \geq \delta$  for all $k$.
This, along with Lemma \ref{lem-low-h}, shows that $E$ is \emph{weakly thick}.

Next, we will prove the conclusion   \noindent $(ii)$. We prove it specifically for the case where $E=(0,\infty)$, and note that the proof for the case where $E=(-\infty,0)$ can be carried out in a similar manner.
For this purpose, we will prove that corresponding to $E=(0,+\infty)$, the sequence
$\{A_{E,n}\}_{n=1}^{+\infty}$ does not satisfy \eqref{equ3.4}.

To this end, we arbitrarily fix $\varepsilon>0$ and $n$ sufficiently large. By \eqref{equ3.1}, we have
\begin{align}\label{equ-C-0}
	J^\varepsilon(\lambda_n)=\{n_0+1,\cdots, n_0+K_n\},
\end{align}
while by \eqref{equ-92-3.1} (with $\alpha=\frac12$) and \eqref{equ2.3}, we see 
\begin{align}\label{equ-C-0.1}
	K_n\approx \lambda_n^{\frac12}\approx n^{\frac13}\rightarrow \infty, \;\;\mbox{as}\;\;n\rightarrow \infty.
\end{align}

Without loss of generality, we can assume that $n_0$ is an even integer. 
Otherwise, when $n_0$  is odd, we remove
the first  row and column, then the smallest eigenvalue
of the remiander does not decrease by the Rayleigh theorem.

Now we shall study the asymptotic behavior of the smallest eigenvalue of $A_{E, n}$. 
We simply write
\begin{align}\label{equ-C-0.2}
	a_{kj}=\int_{0}^{\infty}{\varphi_{n_0+k}(x)\cdot\varphi_{n_0+j}(x)\,\mathrm dx},\;\; 1\leq k,j\leq K_n.
\end{align}
First, we note that each $\varphi_k$ is an odd or even function. Thus, 
\begin{align}\label{equ-C-1}
	a_{kk}=\int_0^{\infty}{|\varphi_{n_0+k}(x)|^2\,\mathrm dx}=\frac12 \;\;\mbox{for all}\;\;k\in \mathbb{N}.
\end{align}
Next, we will derive the expression of each $a_{kj}$ for $k\ne j$. Since $A_{E, n}$ is symmetric,  we only need to consider the case
$k>j$. 
Several facts are as follows: 

\noindent \emph{Fact one:} By \eqref{equ-C-0.2}, \eqref{equ2.5} and \eqref{equ2.13}, we have
\begin{align}\label{equ-C-2}
	a_{kj}&=A_{n_0+k}A_{n_0+j}\int_{0}^{\infty}{Ai(x-\lambda_{n_0+k})\cdot Ai(x-\lambda_{n_0+j})\,\mathrm dx}\nonumber\\
	&=A_{n_0+k}A_{n_0+j}\cdot \frac{Ai(-\lambda_{n_0+k})Ai'(-\lambda_{n_0+j})-Ai(-\lambda_{n_0+j})Ai'(-\lambda_{n_0+k})}{\lambda_{n_0+j}-\lambda_{n_0+k}} .
\end{align}
This, together with \eqref{equ2.2}, implies that
\begin{align}\label{equ-C-2.1}
	a_{kj}=0,\quad\,\,\,\text{if}\,\, k-j \,\,\text{is even}.
\end{align}

\noindent \emph{Fact two:} By \eqref{equ2.6}, \eqref{equ2.11} and the assumption that $n_0$ is even, we have
\begin{equation}\label{equ-C-3}
	A_{n_0+k}=\begin{cases}
		\frac{1}{\sqrt{2\lambda_{n_0+k}}}\cdot |Ai^{-1}(-\lambda_{n_0+k})|,\quad\,\,\,\text{if}\,\, k \,\,\text{is even},\\
		\frac{1}{\sqrt{2}}\cdot |Ai'(-\lambda_{n_0+k})|^{-1},\quad\quad\,\,\,\,\,\,\text{if} \,\,k\,\, \text{is odd}.
	\end{cases}
\end{equation}

\noindent \emph{Fact three:} By \eqref{equ2.4} we have
$$
\frac{\pi}{2}\cdot(k-j)\cdot\lambda^{-\frac12}_{n_0+k}\leq \lambda_{n_0+k}-\lambda_{n_0+j}\leq \frac{\pi}{2}\cdot (k-j)\cdot\lambda^{-\frac12}_{n_0+j}.
$$
This, together with \eqref{equ-C-0}, yields 
\begin{align}\label{equ-C-4}
	\frac{1}{\lambda_{n_0+k}-\lambda_{n_0+j}}=\frac{2}{\pi}\cdot\frac{1}{k-j}\cdot\lambda_n^{\frac12}+O(\lambda_n^{-1}).
\end{align}

Now it follows from \eqref{equ-C-2}--\eqref{equ-C-4} that when $1\le j<k\le K_n$, one has
\begin{equation}\label{equ-C-5}
	a_{kj}=\begin{cases}
		\frac{-1}{\pi(k-j)}\cdot \mbox{sgn}{Ai(-\lambda_{n_0+k})} \cdot \mbox{sgn}{Ai'(-\lambda_{n_0+j})}+O(\lambda_n^{-1}),\,\,\,\,\,\,\,k\,\,\text{ even},\,\,  j\,\, \text{ odd},\\
		\frac{1}{\pi(k-j)}\cdot \mbox{sgn}{Ai'(-\lambda_{n_0+k})}\cdot \mbox{sgn}{Ai(-\lambda_{n_0+j})} +O(\lambda_n^{-1}),\,\,\,\,\,\,\,k\,\, \text{odd},\,\, \,\, j\,\, \text{even},\\
		0, \qquad \qquad \qquad\qquad\qquad\qquad\qquad\qquad\qquad\quad\,\,\,\, k-j \,\,\text{is even}.
	\end{cases}
\end{equation}
where $\mbox{sgn}{f}$ denotes the sign of $f$, for which, we
claim that for any $k\in\mathbb{N}$, one has
\begin{equation}\label{equ-C-6}
	\begin{cases}
		{Ai(-\lambda_{2k})} \cdot {Ai'(-\lambda_{2k+1})} >0,\\
		{Ai'(-\lambda_{2k-1})} \cdot {Ai(-\lambda_{2k})} <0.
	\end{cases}
\end{equation}
In fact, \eqref{equ-C-6} can be verified by the following two cases:

\noindent $(i)$ In the case  $Ai(-\lambda_{2k})<0$, 
recall that the Airy function satisfies the differential equation $Ai''(x)=xAi(x)$, thus we have $Ai''(-\lambda_{2k})>0$. This shows that  $-\lambda_{2k}$ represents a local minimum for $Ai$ near $x=-\lambda_{2k}$. Then there exists a small $\varepsilon>0$, such that $Ai'(x)<0$, when $-\lambda_{2k}-\varepsilon<x<-\lambda_{2k}$ and  $Ai'(x)>0$, when $-\lambda_{2k}<x<-\lambda_{2k}+\varepsilon$.
We can further prove that $Ai'(x)<0$, if $x\in (-\lambda_{2k+2}, -\lambda_{2k})$ and $Ai'(x)>0$, if $x\in (-\lambda_{2k}, -\lambda_{2k-2})$.
In fact, it suffices to prove that the function $Ai'(x)$ does not  change sign within the intervals  $(-\lambda_{2k+2}, -\lambda_{2k})$ and $(-\lambda_{2k}, -\lambda_{2k-2})$.  Suppose, for contradiction,  there exists some $-\lambda_{2k+2}<\mu<-\lambda_{2k}$ or $-\lambda_{2k}<\mu<-\lambda_{2k-2}$ such that $Ai'(-\mu)=0$. Then  \eqref{equ2.2} indicates that $\mu=\lambda_{2k}$ for some $k\in \N^+$, which  is a contradiction.
These imply 
$$
Ai'(-\lambda_{2k+1})<0,\,\,\,\,{Ai'(-\lambda_{2k-1})} >0,
$$
which  implies \eqref{equ-C-6} for this case. 

\noindent $(ii)$ In the case that $Ai(-\lambda_{2k})>0$,   it follows that $-\lambda_{2k}$ represents a local maximum for $Ai$ near $x=-\lambda_{2k}$.
We can use the same approach to that used in $(i)$ to prove that 
$$
Ai'(-\lambda_{2k+1})>0,\,\,\,\,{Ai'(-\lambda_{2k-1})} <0,
$$
which in turn implies \eqref{equ-C-6}.

Now, by  \eqref{equ-C-6} and the assumption that $n_0$ is even, we obtain 
\begin{equation}\label{equ-C-6-1212}
	\begin{cases}
		\mbox{sgn}{Ai(-\lambda_{n_0+k})} \cdot \mbox{sgn}{Ai'(-\lambda_{n_0+j})}&=(-1)^{\frac{k-j+1}{2}}, \,\,\,\,\,\mbox{if}\,\,k\,\text{is even and}\,  j\, \text{is odd},\\
		\mbox{sgn}{Ai'(-\lambda_{n_0+k})}\cdot \mbox{sgn}{Ai(-\lambda_{n_0+j})}&=(-1)^{\frac{k-j-1}{2}}, \,\,\,\,\,\mbox{if}\,\,k\,\text{is odd and}\,  j\, \text{is even}.
	\end{cases}
\end{equation}
Thanks to  \eqref{equ-C-5} and \eqref{equ-C-6-1212}, we have the decomposition
\begin{equation}\label{equ-C-7}
	A_{E, n}=T_{n}+R_{n},
\end{equation}
where  $T_{n}:=(t_{k,j})_{K_n\times K_n}$ is a real symmetric  matrix, with
\begin{equation}\label{equ-C-8}
		t_{kk}=\frac{1}{2}\,\,\,\,\quad\mbox{and}\quad\,\,\,\, t_{kj}=\frac{\sin{\frac{k-j}{2}}\pi}{\pi(k-j)}\,\,\,\,\,\mbox{if}\,\,\, k>j;
	\end{equation}
	and where 
	$R_{n}=\left(r_{ij}\right)_{K_n\times K_n}$ is  a symmetric matrix, with
	\begin{equation}\label{equ-C-9}
		r_{kk}=0,\qquad\,
		|r_{kj}|\leq C\cdot \lambda_n^{-1},\,\,\,\,\,\mbox{if}\,\, k\ne j,
	\end{equation}
	for some  $C>0$  independent of $n, k, j$.
	
	We view $A_{E, n}$ as a small perturbation of $T_{n}$. Our  \textbf{key observation} is that
	$T_n$ is a symmetric Toeplitz matrix:
	\begin{equation}\label{equ-C-10}
		T_n=\frac{1}{\pi}\begin{bmatrix}
			\frac{\pi}{2} & 1  & 0 & \ldots &\frac{\sin{\frac{(K_n-1)\pi}{2}}}{K_n-1} \\\
			1  & \frac{\pi}{2} & 1 & \ddots & \vdots \\\
			0 & 1  & \ddots & \ddots & 0 \\\
			\vdots & \ddots & \ddots  & \frac{\pi}{2} &1 \\\
			\frac{\sin{\frac{(K_n-1)\pi}{2}}}{K_n-1} & \ldots  &0 & 1  & \frac{\pi}{2}
		\end{bmatrix}.
	\end{equation}
	Now we use \eqref{equ-C-8} to determine  the symbol $f(x)$ associated with the matrix $T_n$ (see Section \ref{sec2.2}). 
	More precisely, we consider the infinite series
	$$
	F(z):=\sum_{j=-\infty}^{\infty}a_j\cdot z^j,\qquad z=re^{\i\theta}, \,\,r\in [0, 1),\,\,\theta\in[-\pi, \pi),
	$$
	where $a_j:=t_{i,i+j}$ for $j=1,\dots$ and $a_{-j}=a_j$. Note that the definition of $a_j$
	is independent of $i$ (see \eqref{equ-C-8}). Then by \eqref{equ-C-8}, we find
	\begin{align*}\label{equ-C-11}
		F(z)=\frac{\pi}{2}+\sum_{j=0}^{\infty}{\frac{(-1)^{j}}{2j+1}\cdot\left(z^{2j+1}+z^{-(2j+1)}\right)}
		=\frac{\pi}{2}+\arctan{z}+\arctan{\frac{1}{z}}.
	\end{align*}
	Since
	\begin{equation}\label{equ-C-12}
		\arctan{z}+\arctan{\frac{1}{z}}=\begin{cases}
			\frac{\pi}{2},\quad\quad\,\,\,\,\,\,\,\mbox{if}\,\, \Re z>0,\\
			-\frac{\pi}{2},\quad\quad\,\,\,\mbox{if}\,\, \Re z<0,
		\end{cases}
	\end{equation}
	we have
	\begin{equation}\label{equ-C-13}
		f(\theta)=\lim_{r\rightarrow 1}F(re^{\i\theta})=\begin{cases}
			\pi,\quad\quad\,\,\,\,\,\,\,\mbox{if}\,\,\,\,\theta\in(-\frac{\pi}{2}, \frac{\pi}{2}),\\
			0,\quad\quad\,\,\,\,\,\,\,\mbox{if}\,\,\,\, \theta\in(-\pi, -\frac{\pi}{2})\cup(\frac{\pi}{2}, \pi),
		\end{cases}
	\end{equation}
	However,  it is straightforward to verify  that  if $f(\theta)$ is defined through \eqref{equ-C-13}, then \eqref{equ-C-8}  is exactly the Fourier coefficients of $f$ with $a_{k-j}=t_{kj}$.
	
	\begin{figure}
		\begin{tikzpicture}[scale=0.8]
			\draw[->](-3.4,0)--(3.4,0)node[right,below,font=\tiny]{$\theta$};
			\draw[->](0,-0.2)--(0,2.5)node[right,font=\tiny]{$y$};
			\coordinate (a0) at (-1.5,0);
			\node[below] at (a0) {$-\frac{\pi}{2}$};
			\coordinate (b0) at (1.5,0);
			\node[below] at (b0) {$\frac{\pi}{2}$};
			\node[above=19.2mm,font=\tiny]  at (b0) {$f(\theta)$};
			\coordinate (c0) at (0,1.5);
			\node[above=2.5mm,left=0.02mm] at (c0) {$\frac{\pi}{2}$};
			\coordinate (d0) at (3.1,0);
			\node[below] at (d0) {$\pi$};
			\coordinate (e0) at (-3.1,0);
			\node[below] at (e0) {-$\pi$};
			\draw[color=red, thick,smooth,domain=-3.1:-1.5]plot(\x,0);
			\draw[color=red, thick,smooth,domain=1.5:3.1]plot(\x,0);
			\draw[color=blue, thick,smooth,domain=-1.5:1.5]plot(\x,1.5);
			\draw[color=blue,smooth](0, 1.5)circle(0.03);
			\draw[color=red,smooth](-1.5, 0)circle(0.03);
			\draw[color=red,smooth](1.5, 0)circle(0.03);
			\draw[dashed](-1.5,0) -- (-1.5,1.5);
			\draw[dashed](1.5,0) -- (1.5,1.5);
		\end{tikzpicture}
		\caption{The symbol $f(\theta)$}     \label{fig1}
	\end{figure}
	
	Let $\lambda_{1}(T_n)$ be the smallest eigenvalue of the matrix $T_n$, as defined in \eqref{equ-C-8}. Then by Lemma \ref{lem2.2} and  
	\eqref{equ-C-13}, we have
	\begin{align}\label{equ-C-14}
		\lambda_{1}(T_n)\longrightarrow \inf_{x\in(-\pi,\, \pi)} f(\theta)=0,\quad\,\,\,\mbox{as}\,\, n\rightarrow\infty.
	\end{align}
	
	Next, we consider $R_n$. Let its eigenvalues be
	$$
	\lambda_1(R_n)\leq \cdots\leq \lambda_n(R_n).
	$$
	Since $\lambda_n(R_n)\leq |||R_n|||_{1}:=\max_{j}\sum_{k=1}^{K_n}{|r_{kj}|}$,
	it follows from  \eqref{equ-C-9} and \eqref{equ-C-0.1}  that
	\begin{align}\label{equ-C-15}
		\lambda_n(R_n)\leq K_n\max_{k,j}|r_{kj}|
		\leq C\lambda_n^{\frac12}\cdot \lambda_n^{-1}\leq Cn^{-\frac13}\longrightarrow 0, \quad\,\,\,\mbox{as}\,\, n\rightarrow\infty.
	\end{align}
	Therefore, we obtain that
	\begin{align}\label{equ-C-16}
		0&\leq \lambda_1(A_{E, n})\leq  \lambda_1(T_n)+ \lambda_n(R_n)\longrightarrow 0, \quad\,\,\,\mbox{as}\,\, n\rightarrow\infty,
	\end{align}
	where in the first inequality above, we used the fact that $A_{E, n}$ is non-negative; in the second inequality, we used \eqref{equ-C-7} and applied Weyl's theorem  for the sum of two Hermitian matrices (see e.g. in \cite[p. 238]{HJ}); while in the last step, we used \eqref{equ-C-14} and  \eqref{equ-C-15}.
	
	Finally, by \eqref{equ-C-16}, it follows that \eqref{equ3.4} cannot hold for a uniform constant $\delta>0$. Therefore,  Proposition \ref{prop3.1} indicates that
	for any $T>0$, 
	$E:=(0, \infty)$ is not an \emph{observable set} for \eqref{equ1.1} at $T$. This finishes the proof of the statement $(ii)$ of Theorem \ref{thm1.2}.\qed

		\begin{appendix}
			\renewcommand{\appendixname}{Appendix\,\,}
			\section{Proof of Lemma \ref{lem-low-h}}\label{app-001}

			
			The idea of the proof is taken from  \cite{HWW}. We provide the  proof here for the reader's convenience.
			
			First, we show that $(i)\Rightarrow (ii)$. 
			Suppose that $E$ is \emph{weakly thick}. We will prove $(ii)$ (i.e., \eqref{low-h}) by the following two steps.

			\noindent{\it Step 1. Let $S_k(x):=\frac23(\lambda_k-x)^{\frac32}+\frac{\pi}{4},\;x\in\R$, and let
				\begin{align}\label{equ3.18.1}
					F_{k,\varepsilon}: = \left\{|x|<\frac{\lambda_k}{2}:\,\,  \sin^2{S_k(x)}  \leq \varepsilon  \right\}.
				\end{align}
				We claim that there is  a constant $C_1>0$, independent of $\varepsilon$ and $k$, so that
				\begin{align}\label{equ3-1003-3}
					|F_{k,\varepsilon}|\leq  C_1\sqrt{\varepsilon}\lambda_k\;\;\mbox{for all}\;\;k\geq 1\;\;\mbox{and}\;\; \varepsilon\in (0,1).
			\end{align}}
			For this purpose, we  arbitrarily fix $k\geq 1$ and $\varepsilon\in (0,1)$. Let
			\begin{align}\label{equ3-1003-4}
				J:= \left\{ j\in \mathbb{N}: S_k(\frac{\lambda_k}{2})-\pi\leq j\pi\leq S(0)+\pi\right\}.
			\end{align}
			Several observations are given in order. First, from  the expression of $S_k$,
			we obtain that
			\begin{align}\label{equ-1020-7}
				\sharp J \leq \frac{S_k(0)-S_k(\frac{\lambda_k}{2})}{\pi}+2\leq \frac{1}{3\sqrt{2}}\lambda_k^{\frac{3}{2}}+2.
			\end{align}
			Second, by \eqref{equ3.18.1}, and \eqref{equ3-1003-4},  we have
			\begin{align}\label{equ-1021-1}
				F_{k,\varepsilon}\subset\bigcup_{j\in J'}\left\{|x|<\frac{\lambda_k}{2}:\,\,\,j\pi-\arcsin{\sqrt{\varepsilon}}\leq S_k(x)\leq j\pi+\arcsin{\sqrt{\varepsilon}}\right\}.
			\end{align}
			Third, there is  $C_2>0$ (independent of $k$ and $\varepsilon$) so that when $j\in J$,
			\begin{align}\label{equ3.18.4}
				&\left|\left\{x|<\frac{\lambda_k}{2}:\,\,\,j\pi-\arcsin{\sqrt{\varepsilon}}\leq S_k(x)\leq j\pi+\arcsin{\sqrt{\varepsilon}}\right\}\right|\nonumber\\
				&=2\left(S_k^{-1}(j\pi+\arcsin{\sqrt{\varepsilon}})-S_k^{-1}(j\pi-\arcsin{\sqrt{\varepsilon}})\right)\nonumber\\
				&\leq 4(\arcsin{\sqrt{\varepsilon}})\cdot\sup_{x\in (0, \frac{\lambda_k}{2})}\frac{1}{\sqrt{\lambda_k-x}}\nonumber\\
				&\leq C_2 \sqrt{\varepsilon}  \lambda_k^{-\frac{1}{2}}.
			\end{align}
			In \eqref{equ3.18.4}, for the first equality, Line 2, we used the fact that $S_k(\cdot)$ is continuous and  strictly decreasing on $(0, \frac{\lambda_k}{2})$  and the fact
			that $S_k(x)= S_k(-x)$ with $x\in\R$; for the first inequality, Line 3, we used the rule of the derivative of inverse function and the fact that
			$S_k'(x)=\sqrt{\lambda_k-x}$, with $x\in\R$; for the last inequality, Line 4, we used  the facts
			that $x\in (0, \frac{\lambda_k}{2})$ and  $\arcsin{\sqrt{\varepsilon}}\sim \sqrt{\varepsilon}$.

			According to  \eqref{equ-1020-7}-\eqref{equ3.18.4}, there is $C_1>0$ (independent of $k$ and $\varepsilon$) so that
			$$
			|F_{k,\varepsilon}|\leq  \sharp J  C_2 \sqrt{\varepsilon}  \lambda_k^{-\frac{1}{2}}\leq C_1\sqrt{\varepsilon} \lambda_k\;\;\mbox{for all}\;\;k\geq 1\;\;\mbox{and}\;\; \varepsilon\in (0,1),
			$$
			which leads to \eqref{equ3-1003-3}.

			\noindent {\it Step 2. We prove \eqref{low-h}.}
			
			Several observations are given in order. First, since  $E$ is \emph{weakly thick}, there is $\gamma\in (0, 2]$ and $L>0$ so that
			\begin{align}\label{equ-2-23-1}
				\frac{|E\bigcap [-x,x]|}{x} \geq \gamma\;\;\mbox{for all}\;\;x\geq L.
			\end{align}
			Meanwhile, it follows from \eqref{equ2.1} that there exists  $k_1\in \mathbb{N}$ such that $\lambda_k\ge 2L$ for all $k\ge k_1$. This, along with \eqref{equ-2-23-1}, yields
			\begin{align}\label{equ-1026-0}
				\left|E\bigcap \Big[-\frac{\lambda_k}{2}, \frac{\lambda_k}{2}\Big]\right|\ge \frac{\gamma}{2} \lambda_k,\;\;\mbox{when}\;\; k\ge k_1.
			\end{align}
			Write
			\begin{eqnarray*}\label{equ-1026-1}
				\varepsilon_0:= \Big(\frac{1}{1+ {4C_1}/{\gamma}}\Big)^2 \in (0,1).
			\end{eqnarray*}
			Then it follows from  \eqref{equ3-1003-3} that
			\begin{align}\label{equ-1026-2}
				|F_{k,\varepsilon_0}|\leq C_1\sqrt{\varepsilon_0}\lambda_k\leq \frac{\gamma}4\lambda_k
				\;\;\mbox{for all}\; k\geq  k_1.
			\end{align}
			Combining \eqref{equ-1026-0} and \eqref{equ-1026-2} leads to
			\begin{align}\label{equ3.18.6}
				\left|\Big(E\bigcap \Big[-\frac{\lambda_k}{2}, \frac{\lambda_k}{2}\Big]\Big)\setminus F_{k,\varepsilon_0}\right|\geq \frac{\gamma}{4}\lambda_k,\;\;\mbox{when}\;\; k\geq k_1.
			\end{align}
			Second,  by \eqref{equ3.18.1} (where $\varepsilon=\varepsilon_0$), we have that
			when $k\geq  k_1$,
			\begin{align}\label{equ3.18.7}
				\sin^2{S_k(x)} \geq \varepsilon_0\;\;\mbox{for all}\;\;x\in \Big(E\bigcap \Big[-\frac{\lambda_k}{2}, \frac{\lambda_k}{2}\Big]\Big)\setminus F_{k,\varepsilon_0}.
			\end{align}
			Third,   using  \eqref{equ2.8}, we obtain that
			\begin{multline}\label{equ3.22}
				\int_{E}{|\varphi_k(x)|^2\,\mathrm dx}\ge\int_{E\cap  [-\frac{\lambda_k}{2}, \frac{\lambda_k}{2}]}{|\varphi_k(x)|^2\,\mathrm dx}\\
				\ge \frac{a_k^2}{2}\int_{E\cap [-\frac{\lambda_k}{2}, \frac{\lambda_k}{2}]}{(\lambda_k-|x|)^{-\frac{1}{2}}\sin^2{S_k(x)}\,\mathrm dx}-3a_k^2\int_{E\cap [-\frac{\lambda_k}{2}, \frac{\lambda_k}{2}]}{(\lambda_k-|x|)^{-\frac12}\cdot\,|R_k(x)|^2\mathrm dx}.
			\end{multline}

			Next, we will estimate two terms on the right hand side of \eqref{equ3.22}, with the aid of
			Lemma \ref{lem2.1} and the first two facts above-mentioned.
			First,
			it follows from
			\eqref{equ3.18.7}, \eqref{equ3.18.6} and  \eqref{equ2.9} that when $k\geq k_1$,
			\begin{align}\label{equ3.18.8}
				a_k^2\cdot\int_{E\cap  [-\frac{\lambda_k}{2}, \frac{\lambda_k}{2}]}{(\lambda_k-|x|)^{-\frac{1}{2}}\sin^2{S_k(x)}\,\mathrm dx}
				&\ge a_k^2\cdot\varepsilon_0\cdot\int_{(E\bigcap [-\frac{\lambda_k}{2}, \frac{\lambda_k}{2}])\setminus F_{\varepsilon_0}}{(\lambda_k-|x|)^{-\frac12}\,\mathrm dx}\nonumber\\
				&\ge  a_k^2\cdot\varepsilon_0\cdot\lambda_k^{-\frac12}\cdot\left|(E\bigcap I_k^\delta)\setminus F_{k,\varepsilon_0}\right|\nonumber\\
				&\ge C_3\varepsilon_0\gamma
			\end{align}
			for some $C_3>0$ (independent of $k$). Second,
			it follows from \eqref{equ2.9} that
			\begin{equation}\label{equ3.22.1}
				a_k^2\cdot\int_{E\cap [-\frac{\lambda_k}{2}, \frac{\lambda_k}{2}]}{(\lambda_k-|x|)^{-\frac12}\cdot\,|R_k(x)|^2\mathrm dx} \leq Ca_k^2\cdot \lambda_k\cdot(\frac{\lambda_k}{2})^{-\frac12}\cdot\lambda_k^{-3}
				\leq C_4\lambda_k^{-3}
			\end{equation}
			for some $C_4>0$ (independent of $k$).
			
			Finally, inserting \eqref{equ3.18.8} and \eqref{equ3.22.1} into \eqref{equ3.22}, we find
			that when $k\geq  k_1$,
			\begin{align}\label{equ-1026-4}
				\int_{E}{|\varphi_k(x)|^2\,\mathrm dx}\geq \frac{1}{2}C_3\varepsilon_0\gamma- 3C_4\lambda_k^{-3}.
			\end{align}
			Since $\lambda_k^{-1}\rightarrow 0$ as $k\rightarrow \infty$, we can find $k_2\geq   k_1$ so that
			when $k\geq k_2$,
			\begin{eqnarray*}
				\frac{1}{2}C_3\varepsilon_0\gamma - 3C_4\lambda_k^{-3} \geq \frac{1}{4}C_3\varepsilon_0\gamma,
			\end{eqnarray*}
			which, together with \eqref{equ-1026-4}, leads to \eqref{low-h}.
			
			Next, we show that $(ii)\Rightarrow (i)$.
			Suppose that $E$ is \emph{an observable set at some time $T_0>0$} for \eqref{equ1.1}.
			Then there is  $C_0=C_0(T_0,E)>0$ so that
			\begin{align}\label{equ-10-5-1709}
				\|u_0\|_{L^2}^2\leq C_0\int_0^{T_0}\int_{E}{|e^{-\i tH}u_0|^2\d x\d t}\;\;\mbox{for all}\;\;u_0\in L^2(\mathbb{R}).
			\end{align}
			Let $u_0=\varphi_k$, where $\varphi_k$ is
			the $L^2$ normalized eigenfunction of $H$. Then 
			\begin{align}\label{equ-10-5-1710}
				\int_{E}{|\varphi_k|^2 \d x}\ge C_1:=1/(T_0C_0)  \;\;\mbox{ for all }\;\; k\in \mathbb{N}^+.
			\end{align}
			
			
			In order to apply  Lemma \ref{lem2.1}, we make  the following decomposition:
			\begin{eqnarray}\label{equ3.62}
				\int_{E}|{\varphi_k|^2\,\d x}=I_1+I_2+I_3,
			\end{eqnarray}
			where 
			\begin{align*}
				I_1&:= \int_{E\bigcap\big\{x:|x|< \lambda_k-\delta\big\}}|{\varphi_k(x)|^2\,\mathrm dx}{,}\qquad \,\,I_2:=\int_{E\bigcap\big\{x:\lambda_k-\delta\leq|x|\leq \lambda_k+\delta\big\}}|{\varphi_k(x)|^2\,\mathrm dx}{,}\\
				I_3&:=\int_{E\bigcap\big\{x:|x|>\lambda_k+\delta\big\}}|{\varphi_k(x)|^2\,\mathrm dx}.
			\end{align*}

			To deal with the term $I_1$, we observe that by \eqref{equ2.9}, there is
			$C>0$ (independent of $k$ and $x$) so that
			\begin{align}\label{eq-reminder-es}
				|R_k(x)|\leq  C,
			\end{align}
			when $x$ satisfies either $|x|<\lambda_k-\delta$ or $|x|>\lambda_k+\delta$.
			Now, we write
			\begin{align}\label{equ-10-5-1711}
				I_1=  \int_{{E\bigcap\Big\{x:|x|< \rho\lambda_k\Big\}}}{|\varphi_k|^2\,\mathrm dx}+ \int_{{E\bigcap\Big\{x:\rho\lambda_k<|x|< \lambda_k-\delta\Big\}}}{|\varphi_k|^2\,\mathrm dx}
				:= I_{1,1}+I_{1,2},
			\end{align}
			where $\rho\in(0, 1)$ is some constant to be chosen later.  Notice that for any given $\rho\in(0, 1)$, we have
			$$
			\rho\lambda_k<\lambda_k-\delta,\,\,\,\text{as }\,\,k\to \infty.
			$$
			Thus we can use \eqref{eq-reminder-es}, as well as Lemma \ref{lem2.1}, to find   $k_3\in\mathbb{N}^+$ so that when $k\geq k_3$,
			\begin{align}\label{equ-10-5-1712}
				I_{1,1}\leq  C_2 \int_{E\bigcap\{x:|x|< \rho\lambda_k\}}{\lambda_k^{-\frac12}(\lambda_k-|x|)^{-\frac12}\,\mathrm dx}
				\leq C_2(1-\rho)^{-\frac12}\frac{|E\bigcap\{x:|x|< \rho\lambda_k\}|}{\lambda_k},
			\end{align}
			and
			\begin{align}\label{equ-10-5-1713}
				I_{1,2}\leq  C_3 \int_{{\Big\{x:\rho\lambda_k<|x|< \lambda_k-\delta\Big\}}}{\lambda_k^{-\frac12}(\lambda_k-|x|)^{-\frac12}\,\mathrm dx}
				\leq C_3\int_{\rho}^1{(1-|x|)^{-\frac12}\,\mathrm dx},
			\end{align}
			where $C_2, C_3>0$ are two absolute constants. Since $\int_{\rho}^1{(1-|x|)^{-\frac12}\,\mathrm dx}\rightarrow 0$ as $\rho\rightarrow 1^{-}$,  we can choose $\rho=\rho_0\in (0, 1)$ so that
			\begin{align}\label{equ-10-5-1714}
				\int_{\rho_0}^1{(1-x)^{-\frac12}\,\mathrm dx}<\frac{C_1}{100C_3}.
			\end{align}
			Then  it follows from \eqref{equ-10-5-1711}-\eqref{equ-10-5-1714} that
			\begin{eqnarray}\label{equ3.63}
				I_1
				\leq C_2(1-\rho_0)^{-\frac12}\frac{|E\bigcap\{x:|x|< \rho_0\lambda_k\}|}{\lambda_k}+\frac{C_1 }{100}.
			\end{eqnarray}

			For the term $I_2$, we make use of  Lemma \ref{lem2.1} again to find $k_4\in\mathbb{N}^+$ and $C_4>0$ so that
			\begin{align}\label{equ3.64}
				I_2\leq C_4\lambda_k^{-\frac{1}{2}}\leq \frac{C_1}{100}\;\;\mbox{for all}\;\;k\ge k_4.
			\end{align}

			%
			Similarly, for the term $I_3$,  we obtain from Lemma \ref{lem2.1} that there is some  constant $C_5>0$
			(independent of $k$) so that
			\begin{align}
				I_3
				\leq C_5\lambda_k^{-\frac12} \int_{\lambda_k+\delta}^{\infty}{(x-\lambda_k)^{-\frac12}\exp\left\{-\frac{2}{3}(x-\lambda_k)^{\frac32}\right\}\,\mathrm dx}.
			\end{align}
			By changing variable in the  integral above, we
			can find
			$k_5\in\mathbb{N}^+$ and $C_6>0$ so that
			\begin{eqnarray}\label{equ3.65}
				I_3\leq C_6\lambda_k^{-\frac12}\leq \frac{C_1}{100}\;\;\mbox{for all}\;\;k\ge k_5.
			\end{eqnarray}
			Therefore, it follows from \eqref{equ-10-5-1710}, \eqref{equ3.62}, \eqref{equ3.63}, \eqref{equ3.64}
			and \eqref{equ3.65} that there is  $C>0$, independent of $k$, so that
			\begin{eqnarray}\label{equ3.65-1801}
				\frac{|E\bigcap\{x:|x|<\rho_0\lambda_k\}|}{\lambda_k}\geq C,\;\;\mbox{when}\;\;k\ge k_6:=\max\{k_3,\,k_4,\,k_5\}.
			\end{eqnarray}
			Since $\lambda_k$  satisfies \eqref{equ2.3},  we have
			\begin{align}\label{equ-101-1}
				\varliminf_{\mathbb{N}\ni k\rightarrow  \infty} \frac{|E\bigcap[-\rho_0\lambda_k, \rho_0\lambda_k]|}{\rho_0\lambda_k} = \varliminf_{\mathbb{N}\ni k\rightarrow  \infty} \frac{|E\bigcap[-\rho_1 k^{\frac{2}{3}},\rho_1k^{\frac{2}{3}}]|}{\rho_1k^{\frac{2}{3}}},
			\end{align}
			where
			\begin{align*} \rho_1:=\rho_0\left(\frac{3\pi}{4}\right)^{\frac{2}{3}}.
			\end{align*}
			Meanwhile, from \cite[Lemma 4.4]{HWW}, we see that $E$ is \emph{weakly thick} is equivalent to
			\begin{align*}
				\varliminf_{\mathbb{N}\ni k \rightarrow  \infty}\limits \frac{|E\bigcap [-ak^l,ak^l]|}{ak^l} >0
			\end{align*}
			for some $a>0$ and $l>0$.
			This, along with  \eqref{equ3.65-1801} and \eqref{equ-101-1},
			yields  that $E$ is \emph{weakly thick}.
			
			Hence, we complete the proof of Lemma \ref{lem-low-h}.

			\section{Proof of the claim \eqref{equ-94-10}}\label{app-002}
			
			We first compute
			$$
			\det\mathbb{A}=
			\left|
			\begin{array}{cccc}
				e^{-\i \lambda_1\tau} & e^{-\i \lambda_12\tau}& \cdots & e^{-\i \lambda_1m\tau}\\
				e^{-\i \lambda_2\tau} & e^{-\i \lambda_22\tau}& \cdots & e^{-\i \lambda_2m\tau}\\
				\vdots& \vdots & \cdots & \vdots\\
				e^{-\i \lambda_m\tau} & e^{-\i \lambda_m2\tau}& \cdots & e^{-\i \lambda_mm\tau}
			\end{array}
			\right|.
			$$
			Applying the well-known result  for the Vandermonde determinant, we have
			\begin{align}\label{equ-det-1}
				\det\mathbb{A}=e^{-\i \tau\sum_{1\leq j\leq m}\lambda_j}\prod_{1\leq j<k\leq m}(e^{-\i \lambda_k\tau}-e^{-\i \lambda_j\tau}).
			\end{align}
			A direct computation shows that $|e^{\i x}-1|\geq \frac{1}{2}|x|$ for all $x\in[-1,1]$. Thus
			\begin{align}\label{equ-det-2}
				|e^{-\i \lambda_k\tau}-e^{-\i \lambda_j\tau}|=|e^{-\i (\lambda_k-\lambda_j)\tau}-1|\geq \frac{1}{2}\tau|\lambda_k-\lambda_j|
			\end{align}
			given that $|(\lambda_k-\lambda_j)\tau|\leq \lambda_m\frac{T}{10\lambda_m}<\frac{1}{10}$ when $T\in(0,1)$. It follows from \eqref{equ2.1} and 
			\eqref{equ2.4} that 
			\begin{align}\label{equ-det-3}
				\lambda_k-\lambda_j\geq \lambda_k-\lambda_{k-1}\geq c_1\lambda_k^{-\frac{1}{2}},\;\;\mbox{when}\;\;j<k.
			\end{align}
			
			Now, by \eqref{equ-det-1}-\eqref{equ-det-3}, we see
			\begin{align}\label{equ-det-4}
				|\det\mathbb{A}|\geq \prod_{1\leq j<k\leq m}\frac{c_1\tau}{2}\lambda_k^{-\frac{1}{2}}=(\frac{c_1\tau}{2})^{\frac{m^2}{2}}\prod_{1\leq k\leq m} \lambda_k^{-\frac{1}{2}k}.
			\end{align}
			Recall that $\lambda_k\leq c_2k^{\frac{2}{3}}$ (we assume $c_2>1$), we infer
			\begin{align}\label{equ-det-5}
				\prod_{1\leq k\leq m} \lambda_k^{-\frac{1}{2}k}\geq \prod_{1\leq k\leq m} c_2^{-\frac{1}{2}k}k^{-\frac{1}{3}k}\geq \prod_{1\leq k\leq m} c_2^{-\frac{1}{2}m}m^{-\frac{1}{3}k}=c_2^{-\frac{1}{2}m^2}m^{-\frac{1}{6}m(m+1)}.
			\end{align}
			Note that $\tau=\frac{T}{10m}$ (see \eqref{equ-94-8}), we deduce from \eqref{equ-det-4} and \eqref{equ-det-5} that
			\begin{align}\label{equ-det-6}
				|\det\mathbb{A}|\geq (\frac{c_1T}{20c_2})^{\frac{m^2}{2}}m^{-m^2}.
			\end{align}
			Let $\mathbb{A}^*=(A_{jk})_{1\leq j,k\leq m}$ be the adjugate matrix of $\mathbb{A}$. Since $A_{jk}$ is an $(m-1)\times (m-1)$ matrix and every element has unit modulus length, thus
			\begin{align}\label{equ-det-7}
				|A_{jk}|\leq (m-1)!, \;\;\mbox{when}\;\; 1\leq j,k\leq m.
			\end{align}
			Since $\mathbb{A}^{-1}=(\det \mathbb{A})^{-1}\mathbb{A}^*$, combining \eqref{equ-det-6} and \eqref{equ-det-7} leads to
			\begin{align}\label{equ-det-8}
				\|\mathbb{A}^{-1}\|\leq \left(\sum_{1\leq j,k\leq m} (\frac{A_{jk}}{\det \mathbb{A}})^2 \right)^{\frac{1}{2}}\leq m!(\frac{20c_2}{c_1T })^{\frac{m^2}{2}}m^{m^2}\leq (\frac{c_3m}{T})^{m^2}.
			\end{align}
			This completes the proof of \eqref{equ-94-10}.

		\end{appendix}

		\section*{Acknowledgements}
		S. Huang, G. Wang, and M. Wang are partially supported by the China National Natural Science Foundation under grants Nos. 12171178, 12371450 and 12171442, respectively.
		%

	\end{document}